\documentclass[12pt,leqno]{article}

\usepackage{latexsym,amsmath,amssymb,amsthm,amsfonts,bm}
\usepackage[active]{srcltx}     
\usepackage[all]{xy}
\usepackage{verbatim}
\usepackage{mathrsfs}

\renewcommand{\epsilon}{\varepsilon}

\renewcommand{\mod}{\;\mathrm{mod}\;}

\newcommand{\ring}{\mathfrak{o}}
\newcommand{\C}{\mathbb{C}}
\newcommand{\R}{\mathbb{R}}
\newcommand{\Z}{\mathbb{Z}}
\newcommand{\Q}{\mathbb{Q}}
\newcommand{\B}{\mathbb{B}}

\newcommand{\h}{\mathfrak{H}}
\newcommand{\tr}{\mathrm{tr}}

\newcommand{\res}{\mathrm{res}}
\newcommand{\e}{\mathbf{e}}

\newcommand{\Aa}{\mathcal{A}}
\newcommand{\Bb}{\mathcal{B}}
\newcommand{\Cc}{\mathcal{C}}
\newcommand{\D}{\mathfrak{d}}
\newcommand{\bbb}{\mathfrak{b}}
\newcommand{\ccc}{\mathfrak{c}}
\newcommand{\aaa}{\mathfrak{a}}
\newcommand{\zzz}{\mathfrak{z}}
\newcommand{\Pp}{\mathfrak{P}}
\newcommand{\slz}{\mathrm{SL}_2(\Z)}
\newcommand{\mpz}{\mathrm{Mp}_2(\Z)}

\newcommand{\q}{\mathrm{q}}
\newcommand{\x}{\mathrm{x}}
\newcommand{\y}{\mathrm{y}}
\newcommand{\rt}{\varrho}
\newcommand{\pp}{\pi}
\newcommand{\qq}{\varpi}

\newtheorem{theorem}{Theorem}
\newtheorem{lemma}{Lemma}

\theoremstyle{remark}
\newtheorem{remark}{Remark}

\author{Maryna S. Viazovska}
\begin{document}
\title{Petersson inner products of weight one modular forms}
\maketitle 
\maketitle

\begin{abstract}
In this paper we study regularized Petersson products between a holomorphic theta series associated to a positive definite binary quadratic form and a weakly holomorphic weight 1 modular form with integral Fourier coefficients. In our recent work \cite{Via CM Green} motivated by the conjecture of B. Gross and D. Zagier  on the CM values of higher Green's functions we have discovered that such a Petersson product is equal to the logarithm of a certain algebraic number lying in a ring class field associated to the binary quadratic form. A similar result was obtained independently by W. Duke and Y. Li \cite{Duke}, whose interest in this problem arose from the theory of mock modular forms. The methods of the proof used in \cite{Via CM Green} and \cite{Duke} are quite different. Moreover,
  the approach here gives an explicit factorization formula for the  algebraic number obtained (except for the valuation at the ramified prime).

\end{abstract}

\section{Introduction}\label{sec: intro}
In this paper we study arithmetic properties of the regularized Petersson product between the following two modular forms of weight one: a holomorphic binary theta series  and a weakly holomorphic modular form with integral Fourier coefficients.

More precisely, consider an imaginary quadratic field $K:=\Q(\sqrt{-D})$. For simplicity we assume that $D$ is a prime congruent to $3$ modulo $4$. Let $\Bb$ be an element of the ideal class group $\mathrm{CL}_K$ of $K$. Denote by  $r_\Bb(t)$ the number of integral ideals of norm $t$ in the ideal class $\Bb$. We consider the vector space $\C^{\Z/D\Z}$ with the canonical basis $\{e_\nu|\,\nu\in\Z/D\Z\}$. The theta function \begin{equation}\label{eqn: theta}\Theta_\Bb(\tau)=\sum_{\nu\in\Z/D\Z}e_\nu\!\!\!\!\sum_{t\equiv\nu^2\;(\mathrm{mod}\;D)}
(1+\delta_{0,\nu})\,r_\Bb(t)\,\e\Big(\frac tD\tau\Big)\end{equation} is a holomorphic vector valued modular form of weight 1 and a representation $\rho$ defined in Section \ref{sec: theta}. Here we use the standard notation $\e(x):=e^{2\pi i x}$.

First, consider the classical Petersson inner product between $\Theta_\Bb$ and the cusp form $g_\chi:=\sum_{\Cc\in\mathrm{CL}_K}\chi(\Cc)\Theta_\Cc$ associated to a character $\chi:\mathrm{CL}_K\to\C^\times$. Recall that this product is defined as
 $$(g_\chi,\Theta_\Bb):=\int\limits_{\slz\backslash\h}\langle g_\chi(\tau),\overline{\Theta_\Bb(\tau)}\rangle\,y^{-1}\,dx\,dy, $$
 where $\h=\{\tau\in\C\,|\;\Im(\tau)>0\}$, $\tau=x+iy$, and $\langle e_\mu,e_\nu \rangle=\delta_{\mu,\nu}$.
Applying the Rankin-Selberg method to this integral one can see that the Peterson product can be expressed in terms of Artin $L$-function as \begin{equation}\label{eqn: kronecker}(g_\chi,\Theta_\Bb)=\frac 1h \chi(\Bb^{-1})\,L(\chi_D,1)\,\mathrm{res}_{s=1}L_K(\chi^2,s),\end{equation}
where $h$ is the class number of $K$ and $\chi_D(\cdot)=\big(\frac\cdot D\big).$ Thus, Stark's theorem implies \begin{equation}\label{eqn: unit}(g_\chi,\Theta_\Bb)=\sum_{\Cc\in\mathrm{CL(K)}}\chi^2(\Cc)\log|\epsilon_\Cc|\end{equation} for certain units $\epsilon_\Cc$  in the Hilbert class field of $K$.

   The theory of regularized theta lifts developed by Borcherds, Bruinier, Kudla and others  motivates us to generalize the classical identity \eqref{eqn: unit} by replacing the cusp form $g_\chi$ with a weakly holomorphic modular form. More precisely,  we call $f$ a weakly holomorphic cusp form if it is a weakly holomorphic (vector valued) modular form and has zero constant term. We denote by $S_1^{\,!}(\rho)$ the space of weakly holomorphic cusp forms of weight 1 and representation $\rho$. For $f\in S_1^{\,!}(\rho)$ we define a regularized Petersson product by \begin{equation}\label{eqn: pet regulrized}(f,\Theta_\Bb)^{\mathrm{reg}}:=\lim_{T\to\infty} \int\limits_{\mathcal{F}_T}\langle f(\tau),\overline{\Theta_\Bb(\tau)}\rangle\,y^{-1}\,dx\,dy, \end{equation}
where $$\mathcal{F}_T:=\big\{\tau\in\h\big| -1/2<\Re(\tau)<1/2,\;|\tau|>1,\;\mbox{and}\;\Im(\tau)<T\big\} $$ is the
 truncated fundamental domain of $\slz$. We are interested in the arithmetic properties of the number \eqref{eqn: pet regulrized} when $f$ has integral Fourier coefficients.
In our recent work \cite{Via CM Green} inspired by the conjecture of B. Gross and D. Zagier  on the CM values of higher Green's functions we have found that \begin{equation}\label{eqn: alpha}(f,\Theta_\Bb)^{\mathrm{reg}} =\log|\alpha|\end{equation} for some algebraic number $\alpha$ lying in an abelian extension of $K$. A similar result was obtained independently by W. Duke and Y. Li \cite{Duke}, however their interest in this problem arose from the theory of mock modular forms. The relation between regularized Petersson products and mock modular forms is also studied in \cite{Ehlen}. The main result of the present paper is an explicit factorization formula for the  algebraic number $\alpha$ in \eqref{eqn: alpha} (except for the valuation at ramified primes).

 Before we can state the factorization formula we need to introduce some notations. Recall that $K$ denotes the imaginary quadratic field $\Q\sqrt{-D}$, $h$ be the class number and $H$ be the Hilbert class field of $K$. 
  For a rational prime $p$ with $\big(\frac{p}{D}\big)=-1$ let $\mathcal{P}_p=\{\Pp_i\}_{i=0}^h$ be the set of prime ideals of $H$ lying above $p$. After fixing an embedding $\mathcal{\jmath}:H\to\C$ complex conjugation acts on $H$ and also on the set $\mathcal{P}_p$. Since the class number $h$ is odd, there exists a unique prime ideal in $\mathcal{P}_p$, say $\Pp_1$, with $\Pp_1=\overline{\Pp_1}$. For a prime ideal $\Pp\in\mathcal{P}_p$ there exists a unique element $\sigma\in\mathrm{Gal}(H/K)$ such that \begin{equation}\label{eqn: wp A}\Pp^\sigma=\Pp_1.\end{equation} Denote by $\Aa=\Aa(\Pp)$ the ideal class of  $K$ that corresponds to $\sigma$ under the Artin isomorphism.
\begin{theorem}\label{thm: prime my} Let $\Bb$ be an ideal class of $K$ and let $f\in S_1^{\,!}(\rho)$ be a weakly holomorphic modular form with the Fourier expansion $$f=\sum_{\nu\in \Z/D\Z}e_\nu\,\sum_{\substack{t\in\Z\\t\gg-\infty}} c_\nu(t)\,\e\Big(\frac t D\tau\Big) $$ satisfying $c_0(0)=0$ and $c_\nu(m)\in\Z$ for all $\nu\in \Z/D\Z$. Then there exists an algebraic number $\alpha\in H$ such that $$(f,\Theta_\Bb)^{\mathrm{reg}}=\log|\alpha|. $$
 Moreover, for a rational prime $p$ and a prime $\Pp$ lying above $p$ in the Hilbert class field $H$ we have\\
\begin{equation}\label{eqn: conj3}\mathrm{ord}_{\Pp} (\alpha)=\sum_{t<0}\sum_{\nu\in \Z/D\Z}c_\nu(t)\,r_{\Bb\Aa^2}\Big(\frac{-t}{p}\Big)\,\mathrm{ord}_p(t)\quad\mbox{in the case }\left(\frac{p}{D}\right)=-1,\end{equation}
\begin{equation}\label{eqn: conj3+1}\mathrm{ord}_{\Pp} (\alpha)=0\quad\mbox{in the case }\left(\frac{p}{D}\right)=1.\end{equation}
Here $r_{\Bb\Aa^2}$ is defined as in \eqref{eqn: theta}.
\end{theorem}
\begin{remark}Theorem \ref{thm: prime my} is compatible with, but stronger than the result of J. Schofer \cite{Sch}. More precisely,
Theorem 4.1 
of \cite{Sch} states that the sum over all ideal classes $\Bb\in\mathrm{CL}_K$ of Petersson products $(f,\Theta_\Bb)^{\mathrm{reg}}$ is the logarithm of a rational number and gives a precise factorization formula for this number. In particular, this theorem says and the \emph{sum} over all ideal classes  of the identity \eqref{eqn: conj3} as well as the identity \eqref{eqn: conj3+1} is true.\end{remark}
\begin{remark}Theorem 1 is stated in a conjectural form in \cite{Duke}.\end{remark}
\begin{remark}We don't consider the primes $p$ with $\big(\frac p D\big)$ equal to~$0$. Note that under the assumptions of Theorem 1 the only such prime is $p=D$. Theoretical and numerical evidence suggest that \begin{equation}\label{eqn: conj3 ramified}\mathrm{ord}_{\Pp} (\alpha)=\sum_{t<0}\sum_{\nu\in \Z/D\Z}c_\nu(t)\,r_{\Bb\Aa^2}\Big(\frac{-t}{p}\Big)\,\mathrm{ord}_p(t)\quad\mbox{in the case }\left(\frac{p}{D}\right)=0,\end{equation}
where  $\Pp$ is a prime of $H$ lying above $(\sqrt{-D})$, $\Pp_1$ is the unique prime above $(\sqrt{-D})$ with $\overline{\Pp_1}=\Pp_1$, the ideal class $\Aa$ satisfies $\Pp^{\sigma_\Aa}=\Pp_1$.
\end{remark}
\begin{remark}The space of vector valued modular forms $M_1(\rho)$ is canonically isomorphic to the space $M_1^+(\Gamma_0(D),\chi_D)$ of scalar valued modular forms, see~\cite{Br Bu}. \end{remark}

The idea of the proof of Theorem \ref{thm: prime my} is to use the embedding argument of Theorem 5 in \cite{Via CM Green} and reduce the computation of the regularized integral $(f,\Theta_\Bb)^{\mathrm{reg}}$ to the computation of the local height pairings between certain Heegner points on the modular curve $X_0(D)$.

The paper is organized as follows. In Section \ref{sec: vector-valued} we recall the definition and basic properties of vector valued modular forms. In Sections \ref{sec: theta} and \ref{sec: see-saw} we collect necessary facts from the theory of the Borcherds lift. Also we give a brief review of height theory on the curves in Section \ref{sec: height}. In Section \ref{sec: embedding} we construct a certain meromorphic function $\Psi$ on the modular curve $X_0(D)$. This function has zeroes and poles at the Heegner points and satisfies $(f,\Theta_\Bb)^{\mathrm{reg}}=\log|\Psi(\mathfrak{z})|$ for some CM-point $\mathfrak{z}$. In Section \ref{sec: heegner} we use the computations of the local height pairing made by B.Gross  and D. Zagier in \cite{GZ} and find the local pairings between $\mathfrak{z}$ and $\mathrm{div}(\Psi)$ over the finite places of $H$. This gives as the valuation of $\alpha=\Psi(\mathfrak{z})$ at the primes of $H$ and finishes the proof of Theorem \ref{thm: prime my}. In the last section we illustrate Theorem \ref{thm: prime my} with the following numerical example.

\noindent{\bf Example.\rm}
Consider the prime discriminant $D:=-23$. Recall that the ideal class group of the imaginary quadratic field $K:=\Q(\sqrt{-23})$ has order $3$ and it can be written as $\mathrm{CL}_K=\{\mathcal{J},\mathcal{J}^{-1},\mathcal{O}\}$, where $\mathcal{O}$ denotes the principal ideal class. Consider the modular form
\begin{equation}\label{eqn: f example} f:=23\Big[\frac{E_4E_6}{\Delta},\Theta_\mathcal{J}\Big], \end{equation} where $E_k$ denotes the Eisenstein series of weight $k$, $\Delta$ is Ramanujan's Delta function and $[\cdot,\cdot]$ denotes the Rankin-Cohen brackets. The modular form $f$ belongs to $S^{\,!}_1(\rho)$ and has integral Fourier coefficients.

Consider the following numbers in the Hilbert class field $H$ of $\Q(\sqrt{-23})$. Let $\rt$ be the real root of the polynomial $X^3-X-1$. Define
    \begin{equation}\label{eqn: p_l p_l^2} \pp_{5}=2-\rt,\; \pp_{7}=\rt+2,\; \pp_{11}=2\rt-1,\; \pp_{17}=3\rt+2,\; \pp_{19}=3\rt+1,\end{equation}
$$\qq_{23}=3-\rt,\;\pp_{25}=2\rt^2-\rt+1,\; \pp_{49}=\rt^2-2\rt+3,$$ 
where each $\pp_q$ has norm $q$.
One finds numerically that
  \begin{equation}\label{eqn: example}(f,\Theta_{\mathcal{O}})^{\mathrm{reg}}=\log|\alpha| ,\end{equation} where
  $$\label{eqn: green 23 numerical 234++}
  \qq_{23}^{-23}\,\alpha=\pp_{5}^{\,18}\,\pp_{25}^{-42}\,\pp_{7}^{\,36}\,
\pp_{49}^{-48}\,\pp_{11}^{\,4}\,\pp_{17}^{-22}
\pp_{19}^{-30}\,\rt^{\,207}.$$
In the last section we will prove this identity using Theorem \ref{thm: prime my}.

\section{Lattices and vector valued modular forms}\label{sec: vector-valued}
Recall that the group $\mathrm{SL}_2(\Z)$ has a double cover $\mathrm{Mp}_2(\Z)$, called the \emph{metaplectic group}, whose elements can be written in the form
$$\left( \left( \vcenter{\xymatrix@R=0pt@C=0pt{a&b\\c&d}}\right),\sqrt{c\tau+d}\right) $$
where $\left( \vcenter{\xymatrix@R=0pt@C=0pt{a&b\\c&d}}\right)\in\slz$ and $\sqrt{c\tau+d}$ is one of the two   holomorphic functions of $\tau$ in the upper half-plane whose square is $c\tau+d$. The multiplication is defined so that the usual formulas for the transformation of modular forms of half integral weight work, i. e.  $$ (A,f(\tau))(B, g(\tau))=(AB,f(B(\tau))g(\tau))$$ for $A,B\in \slz$ and $f,g$ are square roots of $c_A\tau+d_A$ and $c_B\tau+d_B$, respectively.

 Let $(V,\q)$ be a rational  quadratic space over $\Q$, that is a rational vector space $V$ equipped with the quadratic form  $\q:V\to\Q$. The corresponding  bilinear form on $V \times V$ is defined by $(x,y)=\frac 12\q(x+y)-\frac 12\q(x-y)$. Suppose that $V$ has signature $(b^+,b^-)$.   Let $L\subset V $ be a lattice. The dual lattice of $L$ is defined as $L'=\{x\in V| (x,L)\subset\Z \}$. We say that $L$ is even if $\q(\ell)\in\Z$ for all $\ell\in L$. In this case $L$ is contained in $L'$ and $L'/L$ is a finite abelian group.

We let the elements $e_\nu$ for $\nu\in L'/L$ be the standard basis of the vector space $\C^{L'/L}$, so that $e_\mu e_\nu=e_{\mu+\nu}$. The complex conjugation acts on $\C^{L'/L}$ by $\overline{e_\mu}=e_\mu$. Consider the scalar product on $\C^{L'/L}$ given by
\begin{equation}\label{eq:green:<,>} \langle e_\mu, e_\nu\rangle =\delta_{\mu,\nu}\end{equation} and extended to $\C^{L'/L}$ by linearity. Recall that there is a unitary representation $\rho_L$ of $\mathrm{Mp}_2(\Z)$ on $\C^{L'/L}$ defined by

\begin{align}\rho_L(\widetilde{T})(e_\nu)&=\e\bigl(\q(\nu)\bigl)\,e_\nu \\
\rho_L(\widetilde{S})(e_\nu)&=i^{(b^-/2-b^+/2)}\,|L'/L|^{-1/2}\,\sum_{\mu\in
L'/L}\e\bigl(-(\mu,\nu)\bigl)e_\mu,
\end{align}
where
\begin{equation}
\label{eqn: generators}\widetilde{T}=\left(\left(\vcenter{\xymatrix@R=0pt@C=0pt{1&1\\0& 1}}\right),1\right)\quad\mbox{and}\quad\widetilde{S}=\left(\left(\vcenter{\xymatrix@R=0pt@C=0pt{0&-1\\1& 0}}\right),\sqrt{\tau}\right)\end{equation} are the standard generators of $\mpz$.

For an integer $n\in\Z$ we denote by $L(n)$ the lattice $L$ equipped with a quadratic form $\q_{n}(\ell):=n\q(\ell)$. In the case $n=-1$ the lattices $L'(-1)$ and $(L(-1))'$ coincide and hence the groups $L'/L$ and $L(-1)'/L(-1)$ are equal. Both representations $\rho_L$ and $\rho_{L(-1)}$ act on $\C^{L'/L}$ and  for  $\gamma\in\mpz$ we have $\rho_{L(-1)}(\gamma)=\overline{\rho_L(\gamma)}$.

A vector valued modular form of half-integral weight $k$ and representation $\rho_L$ is a function $f:\h\to \C^{L'/L}$ that satisfies the following transformation law

\begin{equation}\label{eqn: vec val trans}f\left(\frac{a\tau+b}{c\tau+d}\right)=\sqrt{c\tau+d}^{\,2k}\rho_L\left(\left(\vcenter{\xymatrix@R=0pt@C=0pt{a&b\\c& d}}\right),\sqrt{c\tau+d}\right)f(\tau) .\end{equation}


We will use the notation  $\mathfrak{M}_k(\rho_L)$ for the space of real analytic,  $M_k(\rho_L)$ for the space of holomorphic, and  $M^!_k(\rho_L)$ for the space of weakly holomorphic  modular forms of weight $k$ and representation $\rho_L$. We denote by $S^!_k(\rho_L)$  the space of weakly holomorphic  modular forms of weight $k$ and representation $\rho_L$ with zero constant term at infinity.

Now we recall some standard maps between the spaces of vector-valued modular forms associated to different lattices.

 If $M\subset L$ is a sublattice of finite index, then a vector-valued modular form $f\in \mathfrak{M}_k(\rho_L)$
can be naturally viewed as a vector-valued modular form in $\mathfrak{M}_k(\rho_M)$. Indeed, we have the
inclusions $$M \subset L \subset L' \subset M'$$ and therefore
$$L/M \subset L'/M \subset M'/M.$$
We have the natural map $L'/M \to L'/L$, $\mu\to\bar{\mu}$. The following lemma is proved in  \cite{Br Yang}.
\begin{lemma}\label{res/tr}For $\mathcal{M}=\mathfrak{M},M$  or $M^!$ there are two natural maps
$$\res_{L/M} : \mathcal{M}_k(\rho_L)\to \mathcal{M}_k(\rho_M),$$
and
$$\tr_{L/M} : \mathcal{M}_k(\rho_M) \to \mathcal{M}_k,(\rho_L),$$  given by \begin{equation}\label{res}\bigl(\res_{L/M}(f)\bigl)_\mu=
\left\{\vcenter{\xymatrix@R=0pt@M=5pt{f_{\bar{\mu}},\;\;\mbox{if}\;\mu\in L'/M \\ 0\;\;\mbox{if}\;\mu\notin L'/M}}\right.\qquad\qquad\bigl(f\in \mathcal{M}_k(\rho_L),\;\mu\in M'/M\bigl),\end{equation} and  \begin{equation}\bigl(\tr_{L/M}(g)\bigl)_\lambda=\sum_{\mu\in L'/M:\,\bar{\mu}=\lambda} g_\mu \qquad\qquad\bigl(g\in\mathcal{M}_k(\rho_M),\;\;\lambda\in L'/L\bigl).\end{equation}
\end{lemma}

Now suppose that $M$ and $N$ are two even lattices and $L=M\oplus N$. Then we have $$L'/L\cong (M'/M)\oplus(N'/N).$$ Moreover
$$\C^{L'/L}\cong\C^{M'/M}\otimes\C^{N'/N} $$ as unitary vector spaces and naturally
$$\rho_L=\rho_M\otimes\rho_N. $$
The following lemma can be easily deduced from the unicity of the representation $\rho_M$ and the fact that $\rho_{M(-1)}=\overline{\rho}_M$.
\begin{lemma}\label{lemma:tensor}
For two modular forms $f\in \mathcal{M}_k(\rho_L)$ and $g\in\mathcal{M}_l(\rho_{M(-1)})$ the function $$h:=\langle f, g\rangle_{\C^{M'/M}}=\sum_{\nu\in N'/N} e_\nu\sum_{\mu\in M'/M} f_{\mu\oplus\nu}\,g_\mu $$ belongs to $\mathcal{M}_{k+l}(\rho_N)$.
\end{lemma}
\section{Theta functions and Theta lifts}\label{sec: theta}
In this section we recall the definition of regularized theta lift given by Borcherds in the paper~\cite{Bo1}.

We let $L$ be an even lattice of signature $(2,b^-)$ and let $L'$ be its dual lattice. The positive Grassmannian $\mathrm{Gr}^+(L)$ is
the set of positive definite two-dimensional subspaces $v^+$ of $L\otimes\R$. We write $v^-$ for the orthogonal complement of $v^+$, so that $L\otimes\R$ is the orthogonal direct sum of the positive definite subspace $v^+$ and the negative definite subspace $v^-$. The projections of a vector $\ell\in L\otimes\R$ into the subspaces $v^+$ and $v^-$ are denoted by $\ell_{v^+}$ and  $\ell_{v^-}$, respectively, so that
$\ell=\ell_{v^+}+\ell_{v^-}$.

The vector-valued  Siegel theta function $\Theta_L:\h\times \mathrm{Gr}^+(L)\to\C^{L'/L}$ of $L$ is defined by
\begin{equation}\label{eqn: theta kernel}\Theta_L(\tau,v^+)=y^{b^-/2}\sum_{\lambda\in L'/L}e_\lambda\sum_{\ell\in L+\lambda}\e\bigl(\q(\ell_{v^+})\tau +\q(\ell_{v^-})\bar{\tau}\bigl).\end{equation}
Theorem 4.1 in \cite{Bo1} says that $\Theta_L(\tau,v^+)$ is a real-analytic vector-valued modular form of weight $1-b^-/2$ and representation $\rho_L$ with respect to the variable~$\tau$.

For $f\in \mathfrak{M}_{1-b/2}(\rho_L)$
we define the \emph{regularized theta integral} by
\begin{equation}\label{theta}\Phi_L(v^+,f):=\sideset{}{^\mathrm{reg}}\int
\limits_{\slz\backslash\h}\;\;\langle f(\tau),\overline{\Theta_L(\tau,v^+)}\rangle\, y^{-1-b^-/2}\,dx\,dy \end{equation}
(here the product $\langle,\rangle$ is defined by
$\langle e_\mu,e_\nu \rangle=\delta_{\mu,\nu}$).

The integral is often divergent and has to be regularized. In this paper we consider regularized lifts of weakly holomorphic cusp forms.  In this case the regularization is simpler than in the general situation. For  $f\in S^{\,!}_{1-b/2}(\rho_L)$ we set
$$\Phi_L(v^+,f):=\lim_{T\to\infty}\int_{\mathcal{F}_T}\langle f(\tau),\overline{\Theta_L(\tau,v^+)}\rangle\,y^{-1-b^-/2}\,dx\,dy,$$
where $\mathcal{F}_T$ is a truncated fundamental domain introduced in Section \ref{sec: intro}.

Denote by $\mathrm{Aut}(L)$ the group of those isometries of $L\otimes\R$ that fix each coset of $L$ in $L'$. The action of $\mathrm{Aut}(L)$ on $f$ is given by action on $L'/L$. The regularized integral $\Phi_L(v^+,f)$ is a function on the Grassmannian $\mathrm{Gr}^+(L)$ that is invariant under $\mathrm{Aut}(L)$.

In the case when $L$ has signature $(2,b^-)$ the Grassmanian $\mathrm{Gr}^+(L)$ carries the structure of a Hermitian symmetric space. If $X$
and $Y$ are an oriented orthogonal base of some element $v^+$ of $\mathrm{Gr}^+(L)$, then we map $v^+$ to
the point of the complex projective
space $\mathbb{P}(L\otimes
 \C)$ represented by $Z = X + iY\in L\otimes\C$. The fact that $Z = X + iY$
has norm $0$ is equivalent to saying that $X$ and $Y$ are orthogonal and have the same
norm. This identifies $\mathrm{Gr}^+(L)$  with an open subset of the norm 0 vectors of $\mathbb{P}(L\otimes
 \C)$ in a canonical way, and gives $\mathrm{Gr}^+(L)$ a complex structure invariant under the
subgroup $\mathrm{O}^+(L\otimes \R)$ of index 2 of $\mathrm{O}(L\otimes \R)$ of elements preserving the orientation on the 2
dimensional positive definite subspaces. Thus, the open subset $$\mathcal{P}=\big\{[Z]\in\mathbb{P}(L\otimes\C)\big| (Z,Z)=0\;\mbox{and}\; (Z,\overline{Z})>0\big\}$$
is isomorphic to $\mathrm{Gr}^+(L)$ by mapping $[Z]$ to the subspace $\R\Re(Z)+\R\Im(Z)$.

The following theorem of Borcherds relates the regularized theta lifts associated to the lattice $L$ of signature $(2,b)$ with the infinite products introduced in his earlier paper \cite{Bo95}.

\begin{theorem}\label{thm: borcherds inf pro}(\cite{Bo1}, Theorem 13.3) Suppose that $f\in S^{\,!}_{1-b/2}(\rho_L)$ has the Fourier expansion $$f(\tau)=\sum_{\lambda\in L'/L}\sum_{n\gg-\infty}c_\lambda (n)\,\e(n\tau)\,e_\lambda $$ and the Fourier coefficients $c_\lambda (n)$ are integers for $n\leq 0$. Then there is a meromorphic function $\Psi_L(Z,f)$ on $\mathcal{P}$ with the following properties.\begin{description}\item{1.} $\Psi$ is an automorphic function for the group $\mathrm{Aut}(L,f)$ with respect to some unitary character of $\mathrm{Aut}(L,f)$\item{2.} The only zeros and poles of $\Psi_L$ lie on the rational quadratic divisors $\ell^\bot$ for $\ell\in L$, $\q(\ell)<0$ the order of $l^\bot$ is equal to $$\sum_{\substack{ x\in\R^+\,:\\  xl\in L'}} c_{xl}\bigl(\q(xl)\bigl) $$ \item{3.} $$ \Phi_L(Z,f)=
-4\log|\Psi_L(Z,f)|.$$
\item{4.} One can write an explicit infinite product expansion converging in a neighborhood of each cusp of $\mathrm{Gr}^+(L)$.
    \end{description}
    \end{theorem}
\begin{remark}Theorem 13.3 in \cite{Bo1} is formulated in more general settings and an explicit infinite product expansion is given there.\end{remark}

At the end of this section let us consider the lattices of signature $(2,0)$ in more detail. Recall that there is a one-to-one correspondence between equivalence classes of even lattices of fundamental discriminant $-D$ and ideal classes of the imaginary quadratic field $K=\Q(\sqrt{-D})$. Fix an ideal class $\Bb$. For a fractional ideal $\bbb\in\Bb$ consider an even lattice $N=(\bbb,\q)$, where the quadratic form $\q$ is defined as $\q(x)=\frac{1}{\mathrm{N}_{K/\Q}(\Bb)}\,\mathrm{N}_{K/\Q}(x)$ for $x\in\bbb$. Up to isometry, the lattice $N$ depends only on $\Bb$ and not on the particular choice of $\bbb$. Hence, we denote this lattice by $N_\Bb$. Moreover, the representation $\rho_N$ defined in Section~\ref{sec: vector-valued} depends only on the genus of $N$. Thus, for $D$ prime the representation $\rho_N$ is the same for all fractional ideals of $K$ and we denote this representation by $\rho$. The theta function $\Theta_\Bb$ defined in Section~1 coincides  with the theta function $\Theta_N$ defined by \eqref{eqn: theta kernel}.  It follows from the definition of theta lift that for $f\in S_1^{\,!}(\rho)$
$$\Phi_N(f)=(f,\Theta_\Bb)^{\mathrm{reg}}. $$
Thus, the regularized Petersson product \eqref{eqn: pet regulrized} can be seen as a theta lift of $f$ to a zero-dimensional Grassmanian $\mathrm{Gr}^+(N)$. From this point of view Theorem \ref{thm: prime my} is an ``arithmetic analog'' of Theorem \ref{thm: borcherds inf pro}. This approach is studied in the upcoming work of S. Ehlen \cite{Ehlen1}.  

\section{A see-saw identity}\label{sec: see-saw}
In the paper \cite{Ku84} S. Kudla  introduced the notion of a ``see-saw dual reductive pair''. It gives rise to a family of identities between inner products of automorphic forms on different groups, thus clarifying  the source of identities of this type which appear  in many places in the literature, often obtained from complicated manipulations. Here we prove a see-saw identity for the regularized theta integrals described in the previous section. This identity will play a central role in our proof of Theorem \ref{thm: prime my}.

 Suppose that $(V,\q)$ is a rational quadratic space of signature $(2,b)$ and $L\subset V $ is an even lattice. Let $V=V_1\oplus V_2$ be the rational orthogonal splitting of $(V,\q)$ such that the space $V_1$ has the signature $(2,b-d)$ and the space $V_2$ has the signature $(0,d)$. Consider two lattices $N:=L\cap V_1$ and $M:=L\cap V_2$. We have two orthogonal projections
 $$ \mathrm{pr}_M:L\otimes\R\to M\otimes\R \quad\mbox{and}\quad\mathrm{pr}_N:L\otimes\R\to N\otimes\R.$$
 Let $M'$ and $N'$ be the dual lattices of $M$ and $N$. 
 We have the following  inclusions
 $$ M\subset L,\quad N\subset L,\quad M\oplus N\subseteq L\subseteq L'\subseteq M'\oplus N', $$ and equalities of the sets $$\mathrm{pr}_M(L')=M',\quad\mathrm{pr}_N(L')=N'. $$

  Consider the rectangular $|L'/L|\times|N'/N|$ dimensional matrix $T_{L,N}$ with entries $$\vartheta_{\lambda,\nu}(\tau)=\sum_{\xymatrix@R=0pt@C=0pt@M=1pt{\scriptstyle m\in M': \\ \scriptstyle  m+\nu\in\lambda+L}}\e\bigl(-\q(m)\tau\bigl),$$ where $ \lambda\in L'/L, \nu\in N'/N, \tau\in\h .$
  This sum is well defined since $N\subset L$. Note that the lattice $M$ is negative definite and hence the  series converges absolutely.

   For  a function $f=(f_\lambda)_{\lambda\in L'/L}\in M_{k+d/2}(\rho_N)$ we define $g=(g_\nu)_{\nu\in N'/N}$  by
\begin{equation}\label{T_L,N} g_\nu(\tau)=\sum_{\lambda\in L'/L}\vartheta_{\lambda,\nu}(\tau)\,f_\lambda(\tau).\end{equation}
In other words \begin{equation}\label{T_L,N 1}g=T_{L,N}f ,\end{equation}
where $f$ and $g$ are considered as column vectors.

\begin{theorem}Suppose that the lattices $L$, $M$ and $N$  and the functions $f$, $g$ are defined as above. Then the function $g$ belongs to $M_{k+d/2}(\rho_N)$.
Thus,  there is a map $T_{L,N}:M_k(\rho_L)\to M_{k+d/2}(\rho_N)$ defined by \eqref{T_L,N 1}.
\end{theorem}
\begin{proof} Consider the function $$\Theta_{M(-1)}(\tau)=\overline{\Theta_M(\tau)}=\sum_{\mu\in M'/M} e_\mu\sum_{m\in M+\mu}\e(-\q(m)\tau),$$ which belongs to $M_{d/2}(\rho_{M(-1)})$. It follows from \eqref{T_L,N} and \eqref{res} that $$T_{L,N}(f)= \big\langle \mathrm{res}_{L/{M\oplus N}}(f), \Theta_{M(-1)}\big\rangle_{\C^{M'/M}}.$$
Thus, from Lemma \ref{lemma:tensor} we deduce that $T_{L,N}(f)$ is in $M_{k+d/2}(\rho_N)$.\end{proof}

\begin{theorem}\label{th see-saw borcherds}Let $L$, $M$, $N$ be as above. Denote by  $i:\mathrm{Gr}^+(N)\to \mathrm{Gr}^+(L)$  the natural embedding induced by the inclusion $N\subset L$.  Then for $v^+\in \mathrm{Gr}^+(N)$ and the theta lift of a function $f\in M_{1-b/2}^{\,!}(\rho_L)$  the following holds \begin{equation} \label{see-saw}\Phi_L(i(v^+),f)=\Phi_N(v^+,T_{L,N}(f)).\end{equation}
\end{theorem}
\begin{proof}
For a vector $\ell\in L'$ denote $m=\mathrm{pr}_M(\ell)$ and $n=\mathrm{pr}_N(\ell)$. Recall that $m\in M'$ and $n\in N'$. Since $v^+$ is an element of $\mathrm{Gr}^+(N)$  it is orthogonal to $M$. We have $$\q(\ell_{v^+})=\q(n_{v^+}),\quad\q(\ell_{v^-})=\q(m)+\q(n_{v^-}).$$ Thus for $\lambda\in L'/L$ we obtain $$\Theta_{\lambda+L}(\tau,v^+)=\sum_{\ell\in\lambda+L}\e\bigl( \q(\ell_{v^+})\tau+ \q(\ell_{v^-}\bigl)\bar{\tau}) $$
$$=\sum_{\xymatrix@R=0pt@C=0pt@M=1pt{\scriptstyle m\in M',\, n\in N': \\ \scriptstyle m+n\in\lambda+L}}\e\bigl( \q(n_{v^+})\tau+ \q(n_{v^-})\bar{\tau}+\q(m)\bar{\tau}\bigl). $$
Since $N\subset L$ we can rewrite this sum as
$$ \Theta_{\lambda+L}(\tau,v^+)=\sum_{\nu\in N'/N}\Theta_{\nu+N}(\tau,v^+)\,\overline{\vartheta_{\nu,\lambda}(\tau)}.$$
Thus, we see that for $f=(f_\lambda)_{\lambda\in L'/L}$ the following scalar products are equal
$$\langle f,\overline{\Theta_L(\tau,v^+)}\rangle =\langle T_{L,N}(f),\overline{\Theta_N(\tau,v^+)}\rangle .$$ Therefore, the regularized integrals \eqref{theta} of both sides of the equality are also equal.  \end{proof}

\section{Local and global heights on curves}\label{sec: height}
In this section we review the basic ideas of N\'eron's theory. A more detailed overview of this topic is given in \cite{Gross}. Let $X$ be a
non-singular, complete, geometrically connected curve over the locally compact
field $k_v$. We normalize the valuation map $|\;|_v:k_v^*\to\R_+^\times$ so that for any Haar
measure $dx$ on $k_v$ we have the formula $\alpha^*(dx) = |\alpha|_v\cdot dx$.

 Let $a$ and $b$ denote divisors
of degree zero on $X$ over $k_v$ with disjoint support. Then N\'eron defines a \emph{local symbol} $\langle a, b\rangle _v$ in $\R$  which is \begin{description} \item{(i)} bi-additive, \item{(ii)} symmetric, \item{(iii)} continuous, \item{(vi)} satisfies the property
$\langle \sum m_x (x), (f) \rangle _v=\log |\prod f(x)^{m_x}|_v$, when $b = (f)$ is principal.
\end{description} These properties characterize
the local symbol completely.

When $v$ is archimedean, one can compute the N\'eron symbol as follows.
Associated to $b$ is a Green's function $G_b$ on the Riemann surface $X(\overline{k_v})-|b|$ which
satisfies $\partial\overline{\partial}G_b=0$ and has logarithmic singularities at the points in $|b|$. More
precisely, the function $G_b-\mathrm{ord}_z(b)\log|\pi|_v$, is regular at every point $z$, where $\pi$ is a
uniformizing parameter at $z$. These conditions characterize $G_b$ up to the addition of
a constant, as the difference of any two such functions would be globally harmonic.
The local formula for $a = \sum m_x (x)$ is then
$$(a, b)_v = \sum m_x\,G_b(x).$$
This is well-defined since $\sum m_x=0$ and satisfies the required properties since if
$b=(f)$ we could take $G_b=\log|f|$.

If $v$ is a non-archimedean place, let $\ring_v$ denote the valuation ring of $k_v$ and $q_v$ the
cardinality of the residue field. Let $\mathcal{X}$ be a regular model for $X$ over $\ring_v$ and extend
the divisors $a$ and $b$ to divisors $A$ and $B$ of degree zero on $\mathcal{X}$. These extensions are
not unique, but if we insist that $A$ has zero intersection with each fibral
component of $\mathcal{X}$ over the residue field, then the intersection product $(A\cdot B)$ is well defined.
We have the formula
$$\langle a, b\rangle_v= -(A\cdot B) \log q_v.$$

Finally, if $X$, $a$, and $b$ are defined over the global field $k$ we have $(a, b)_v = 0$ for
almost all completions $k_v$ and the sum
\begin{equation}\label{eqn: global height}\langle a, b\rangle = \sum_{\mathrm{places}\;v} \langle a, b\rangle_v\end{equation}
depends only on the classes of $a$ and $b$ in the Jacobian. This is equal to the
global height pairing of N\'eron and Tate.

It is desirable to have an extension of the local pairing to divisors $a$ and $b$ of degree $0$ on $X$ which are not relatively prime. At the loss of some functoriality, this is done in \cite{Gross} as follows.

 At each point $x$ in the common support, choose a basis $\frac{\partial}{\partial t}$ for the tangent space and let $\pi$ be a uniformizing parameter with $\frac{\partial \pi}{\partial t}=1$. Any function $f\in k_v(X)^*$ then has a well-defined ``value'' at $x$:
 $$f[x]=\frac{f}{z^m}(x) \;\mbox{in}\;k_v^*,$$
 where $m=\mathrm{ord}_x f$. This depends only on $\frac{\partial}{\partial t}$, not on $\pi$. Clearly we have $$fg[x]=f[x]\,g[x]. $$
 To pair $a$ with $b$ we may find a function $f$ on $X$ such that $b=\mathrm{div}(f)+b'$, where $b'$ is relatively  prime to $a$. We then define
 \begin{equation}\label{eqn: common support} \langle a, b\rangle_v=\log|f[a]|_v+\langle a, b' \rangle.\end{equation}
 This definition is independent of the choice of $f$ used to move $b$ away from $a$.
 The same decomposition formula \eqref{eqn: global height} into
local symbols can be used even when the divisors $a$ and $b$ have a common support provided that the uniformizing parameter $\pi$ at each point of
their common support is chosen over $k$.

\section{Embedding argument}\label{sec: embedding}
Recall that $K$ denotes the imaginary quadratic field $\Q(\sqrt{-D})$. In this section  we construct, for each $\Bb\in\mathrm{CL}_K$ and each $f\in S^{\,!}_1(\rho)$, a meromorphic function $\Psi$ on the modular curve $X_0(D)$ that satisfies the following two properties: the divisor of this function is supported on Heegner points; the identity \begin{equation}\label{eqn: Psi}(f,\Theta_\Bb)^{\mathrm{reg}}=\log|\Psi(\mathfrak{z})|\end{equation} holds for some CM-point $\mathfrak{z}$. The main tools we use in this section are Borcherds lifts and see-saw identities introduced in Section \ref{sec: theta}.

 Our first goal is to find a convenient lattice of signature $(2,1)$ that contains the positive definite lattice associated to the ideal class $\Bb$ as a lower rank sublattice. To this end we consider the lattice \begin{equation}\label{eqn: lattice L}L=\bigg\{\left(\vcenter{\xymatrix@R=0pt@C=2pt@W=0pt{A/D&B\\B&C}}\right) \Big|A,B,C\in\Z \bigg\}\end{equation} equipped with the quadratic form $\q(x):=-D\det(x)$. Its dual lattice $L'$ is given by
\begin{equation}\label{eqn: lattice L}L'=\bigg\{\left(\vcenter{\xymatrix@R=0pt@C=2pt@W=0pt{A'/D&B'/2D\\B'/2D&C'}}\right) \Big|A',B',C'\in\Z \bigg\}.\end{equation}

For $\ell\in L'$ with $\q(\ell)<0$ denote by $\mathfrak{z}_\ell$ the point in $\h$ corresponding to the positive definite subspace $\ell^\bot$ via \eqref{eqn: D z to v+}.  More explicitly, for the vector
$$ \ell=\left(\vcenter{\xymatrix@R=0pt@C=0pt@W=0pt{\gamma&-\beta/2\\ -\beta/2& \alpha}}\right)$$  the point $\mathfrak{z}_\ell$ is a root of the quadratic equation \begin{equation}\label{eqn: frakz}\alpha\mathfrak{z}_\ell^2+\beta\mathfrak{z}_\ell+\gamma=0.\end{equation}
Let us recall the following standard facts about lattices and fractional ideals.
\begin{lemma}\label{lemma: standard ccc} Suppose that  $-D<0$ is a square-free discriminant and $[a,b,c]$ is a primitive quadratic form of disctiminant $-D$. Let $\mathfrak{z}$ be a solution of the equation $a\mathfrak{z}^2+b\mathfrak{z}+c=0$. Then the lattice  $$\ccc=\Z+\mathfrak{z}\Z$$ is a fractional ideal of the imaginary quadratic field $K=\Q(\sqrt{-D})$. Moreover, this ideal satisfies
\begin{equation}\label{eqn: bb'}\ccc\overline{\ccc}=(a)^{-1}.\end{equation}
\end{lemma}

\begin{lemma}\label{lemma: standard ccc dual}Let $\ccc\subset K$ be a fractional ideal.  Consider the quadratic form $\q(\cdot)$ on $K$  given by $\q(\beta)=N_{K/\Q}(\beta)$.  Then the dual lattice of $\ccc$ with respect to this quadratic form  is equal to $(N_{K/\Q}(\ccc))^{-1}\,\ccc\,\mathfrak{d}^{-1}$. Here $\D$ denotes the different of $\ring_K$, i. e. the principal ideal $(\sqrt{-D})$.
\end{lemma}

The following two lemmas are crucial to show that for each ideal class $\Bb$ the associated positive definite lattice is contained in $L$ as a lower rank sublattice.
 \begin{lemma}\label{lemma: m Cc} For each ideal class $\Cc\in\mathrm{CL}_K$ there exists a vector $m\in L'$  such that $\q(m)=-1/4$ and  $\mathfrak{z}_m\Z+\Z\subset K$ is a fractional ideal in $\Cc$.
 \end{lemma}
 \begin{proof} The classical correspondence between fractional ideals of $\ring_K$ and binary quadratic forms of discriminant $-D$ implies that for each ideal class $\Cc\in\mathrm{CL}_K$ there exist $A,B,C\in\Z$ such that
 $$B^2-4AC=-D$$ and for $\mathfrak{z}\in\h$ satisfying
 $$A\mathfrak{z}^2+B\mathfrak{z}+C=0 $$ the subset  $\mathfrak{z}\Z+\Z$ of $K$ is a fractional ideal in the ideal class $\Cc$. Or equivalently, there exists a half-integral matrix $$l=\left(\vcenter{\xymatrix@R=0pt@C=2pt@W=0pt{C&-B/2\\-B/2&A}}\right) $$ with $$\mathfrak{z}_l\Z+\Z\in\Cc.$$ For each $x\in \slz$ the fractional ideal $\mathfrak{z}_{xlx^t}\Z+\Z$ is equivalent to $\mathfrak{z}_l\Z+\Z$. It is easy to see, that the matrix $l$ is $\slz$-equivalent to some matrix of the form $$\tilde{l}=\left(\vcenter{\xymatrix@R=0pt@C=2pt@W=0pt{\tilde{C}&-\tilde{B}/2\\ -\tilde{B}/2&\tilde{A}}}\right),\quad \tilde{A}\in D\Z, \tilde{B}\in D\Z, \tilde{C}\in\Z.$$
 Then the matrix $m=:\tilde{l}/D$ belongs to $L'$, has norm $-1/4$, and since $\mathfrak{z}_m=\mathfrak{z}_{\tilde{l}}$ the fractional ideal $\mathfrak{z}_m\Z+\Z$ belongs to the ideal class $\Cc$. Lemma is proved.
 \end{proof}

 \begin{lemma}\label{lemma: L cap m bot} Let $m\in L'$ be a vector of the norm $-1/4$. Set $N:=L\cap m^\bot$. Denote by $\ccc$ the fractional ideal $\frak{z}_m\Z+\Z$. Then the following holds\\
 (i) the lattice $N$ is isomorphic to the fractional ideal $\ccc^2$ equipped with the quadratic form $\q(\gamma)=\frac{1}{N_{K/\Q}(\ccc^2)}N_{K/\Q}(\gamma)$;\\
 (ii) $L= N\oplus 2m\Z$.
 \end{lemma}
 \begin{proof} First we prove part (i). Each element of $L'$ can be written as $$m= \frac 1D \left(\vcenter{\xymatrix@R=0pt@C=0pt@W=0pt{c&-b/2\\-b/2& a}}\right)$$ for some $a\in D\Z,b\;\in\Z,\;c\in\Z$. The condition $4D\q(m)=b^2-4ac=-D$ implies that $b\in D\Z$. Set $$Z:= \frac{a}{D}\, \left(\vcenter{\xymatrix@R=0pt@C=0pt@W=0pt{\mathfrak{z}_m^2&\mathfrak{z}_m\\\mathfrak{z}_m& 1}}\right).$$ This element of $L\otimes\C$ satisfies
 $$\q(Z)=\q(\overline{Z})=0\quad \mbox{and}\quad (Z,\overline{Z})=1.$$ Moreover,  the elements $Z$ and $\overline{Z}$ are both orthogonal to $m$.
 Consider the map $$\imath:K\to N\otimes\Q$$ defined by
 $$s\to \overline{s} Z+s\overline{Z}.$$
  This map is an isometry, assuming that the quadratic form on $K$ is given by $\q(\beta)=N_{K/\Q}(\beta)$ and the quadratic form on $N\otimes\Q$ is given by $\q(\ell)=-D\det(\ell)$. We have
 \begin{align*}\imath(a)\,=\,&\frac aD\left(\vcenter{\xymatrix@R=0pt@C=0pt@W=0pt{\zzz_m^2+\overline{\zzz_m}^2&\zzz_m+\overline{\zzz_m}\\ \zzz_m+\overline{\zzz_m}& 2}}\right)
 \,=\,\frac 1D\left(\vcenter{\xymatrix@R=0pt@C=0pt@W=0pt{(b^2-D)/2&-ab\\-ab& 2a^2}}\right),\\[5 pt]
 \imath(a\mathfrak{z}_m)
 \,=\,&\frac aD
 \left(\vcenter{\xymatrix@R=0pt@C=0pt@W=0pt{\zzz_m\overline{\zzz_m}(\zzz_m+\overline{\zzz_m})&
 \zzz_m^2+\overline{\zzz_m}^2\\ \zzz_m^2+\overline{\zzz_m}^2& \zzz_m+\overline{\zzz_m}}}\right)
 \,=\,\frac 1D\left(\vcenter{\xymatrix@R=0pt@C=0pt@W=0pt{-bc&(b^2-D)/2\\(b^2-D)/2& ab}}\right),\\[5pt]
 \imath(a\mathfrak{z}_m^2)
 \,=\,&\frac aD
 \left(\vcenter{\xymatrix@R=0pt@C=0pt@W=0pt{\zzz_m^2\overline{\zzz_m}^2&
 \zzz_m\overline{\zzz_m}(\zzz_m+\overline{\zzz_m})\\ \zzz_m\overline{\zzz_m}(\zzz_m+\overline{\zzz_m})& \zzz_m^2+\overline{\zzz_m}^2}}\right)
 \,=\,\frac 1D\left(\vcenter{\xymatrix@R=0pt@C=0pt@W=0pt{2c^2&-bc\\-bc& (b^2-D)/2}}\right).  \end{align*}
 Using  that $a,b\in D\Z$ and $b\equiv D(\mathrm{mod}\;2)$, we see that
 \begin{equation}\label{eqn: L cap m^bot subset} \imath(a\,\Z+a\,\mathfrak{z}_m\,\Z+a\,\mathfrak{z}_m^2\Z)\subseteq N.\end{equation}
 On the other hand
 \begin{equation}\label{eqn: Cc^2}(a)\ccc^2=a\Z+a\mathfrak{z}_m\Z+a\mathfrak{z}_m^2\Z.\end{equation}
 Lemma \ref{lemma: standard ccc} implies that the ideal $(a)\ccc^2$ has norm 1. Hence, by Lemma \ref{lemma: standard ccc dual} the dual lattice of $(a)\ccc^2$ in $K$ is equal to $\D^{-1}(a)\ccc^2$. Since $\imath$ is an isometry, the dual of $\imath((a)\ccc^2)$ is $\imath(\D^{-1}(a)\ccc^2)$. We have the inclusions
 \begin{equation}\label{eqn: inclusions}\imath((a)\ccc^2)\subseteq N\subset N'\subseteq \imath(\D^{-1}(a)\ccc^2).\end{equation}
 Since $(\D^{-1}(a)\ccc^2)/((a)\ccc^2)\cong\Z/D\Z$ we find that $|N'/N|$ is a divisor of $D$. Since a positive definite 2-dimensional even lattice can not be unimodular, we deduce that  $|N'/N|=D$. Thus the symbols ``$\subseteq$'' in \eqref{eqn: inclusions} should be replaced by ``$=$''. Part (i) of Lemma \ref{lemma: L cap m bot} is proved.

 Now we prove (ii). The condition $\q(m)=-1/4$ implies that $b\in D\Z$. Hence, the element $2m$ belongs to $L$. Set $M:=2m\Z$. We have the following inclusions
 $$M'\oplus N'\subseteq L'\subset L\subseteq M\oplus N .$$
 Observe that
 $$|L'/L|=2D,\quad |M'/M|=2,\quad |N'/N|=D.$$
 Thus, $L=M\oplus N$ and $L'=M'\oplus N'$.
 \end{proof}
  We combine the previous two lemmas in the following theorem.
\begin{theorem}\label{thm: L cap m bot} For each ideal class $\Bb$ of $K$ there exists a vector $m\in L'$ such that\\
(i) $\q(m)=-1/4$;\\
(ii) the lattice $N:=L\cap m^\bot$ is isomorphic to the lattice $N_\Bb$ defined in Section~\ref{sec: theta};\\
(iii) $L=N\oplus2m\Z$.
\end{theorem}
\begin{proof} Since $D$ is prime, the class number of $K$ is odd. Thus, each ideal class $\Bb$ is equal to $\Cc^2$ for some $\Cc\in\mathrm{CL}_K$. Let $m\in L'$ be the vector constructed in Lemma \ref{lemma: m Cc}, which satisfies $\mathfrak{z}_m\Z+\Z\in\Cc$. Then Lemma \ref{lemma: L cap m bot} readily implies that $m$ satisfies the conditions of the theorem.
\end{proof}
Our next goal is to find a preimage of a function $f\in M_{1}^{\,!}(\rho_N)$ under the map $T_{L,N}$ defined in  Theorem
\ref{th see-saw borcherds}. To this end we show that the map $T_{L,N}:M_{1/2}^!(\rho_L)\to M_1^!(\rho_N)$ is surjective. Moreover, the kernel of $T_{L,N}$ is infinite dimensional and using its elements we can find a preimage of $f$ that satisfies additional conditions on Fourier coefficients. This additional conditions are needed to simplify the computation of CM-values of the theta lift $\Phi_L(z,f)$, which we will make in Section \ref{sect: proof of prime my}.

Before we proceed with the proof let as recall some properties of the Fourier coefficients of elements in $M^!_k(\rho_N)$ and $M^!_k(\rho_L)$. Recall that $N'/N\cong\Z/D\Z$ and $L'/L\cong\Z/2D\Z$. Moreover, we can choose isomorphisms $\imath_N:\Z/D\Z\to N'/N$ and $\imath_L:\Z/2D\Z\to L'/L$ such that $\q(\imath_N(\nu))\equiv\nu^2/D(\mathrm{mod}\;\Z)$ for each $\nu\in\Z/D\Z$ and
$\q(\imath_L(\lambda))\equiv\lambda^2/4D(\mathrm{mod}\;\Z)$ for each $\lambda\in\Z/2D\Z$.
Suppose that $f=(f_\nu)_{\nu\in N'/N}$ belongs to $M^{\,!}_k(\rho_N)$ for some  odd $k$. Then, the transformation property \eqref{eqn: vec val trans} for the elements $\widetilde{R}=\left(\left(\vcenter{\xymatrix@R=0pt@C=0pt{-1&0\\0&-1}}\right),i\right)$ and $\widetilde{T}= \left(\left(\vcenter{\xymatrix@R=0pt@C=0pt{1&1\\0& 1}}\right),1\right)$ of $\mpz$ implies that $$ f_\nu=f_{-\nu},\quad\nu\in \Z/D\Z$$ and $f_\nu$ has the Fourier expansion of the form $$ f_\nu=\sum_{t\equiv\nu^2(\mathrm{mod}\;D)}c(t)\,\e\Big(\frac{t}{D}\tau\Big).$$ Therefore, since $D$ is prime the Fourier expansion of $f$ can be written as
$$f(\tau)=\sum_{\nu\in\Z/D\Z}e_\nu \sum_{t\equiv\nu^2(\mathrm{mod}\;D)}c(t)\,\e\Big(\frac{t}{D}\tau\Big).$$
Similarly, we see that for $l\in1/2+2\Z$ each modular form $h\in M^{\,!}_l(\rho_L)$  has the Fourier expansion of the form
$$h(\tau)=\sum_{\lambda\in\Z/2D\Z}e_\lambda \sum_{d\equiv\lambda^2(\mathrm{mod}\;D)}b(d)\,\e\Big(\frac{d}{4D}\tau\Big).$$
 \begin{theorem}\label{thm: trick pro} Let the  lattices $L,\;N$, and the vector $m$ be as in Theorem \ref{thm: L cap m bot}. Suppose that $f\in M_{1}^{\,!}(\rho_N)$ is a modular form with zero constant term and rational Fourier coefficients. Then there exists a function $h\in S_{1/2}^{\,!}(\rho_L)$ such that: \begin{description}\item{(i)} the function $h(\tau)=\sum_{\lambda\in \Z/2D\Z}e_\lambda\sum_{d\equiv\lambda^2(\mathrm{mod}\;4D)}b(d)\,\e(\frac{d}{4D}\tau)$ has rational Fourier coefficients; \item{(ii)} the Fourier coefficients  of $h$ satisfy $b(-Ds^2)=0$ for all $s\in\Z$; 
 \item{(iii)} $T_{L,N}(h)=f$.
\end{description}
\end{theorem}
\begin{proof}
Denote by $S$ the lattice $\Z$ equipped with the quadratic form $\q(x):=-x^2$. For this lattice  we have $S'/S\cong \Z/2\Z$. Lemma \ref{lemma: L cap m bot} implies that $L\cong N\oplus S$.
Note that $L'/L\cong S'/S\times N'/N$ and $\rho_L=\rho_S\otimes\rho_N$.
Set
$$\theta_0(\tau,z)=\sum_{n\in\Z}\e(n^2\tau+2nz) ,\quad \theta_1(\tau,z)=\sum_{n\in\frac 12+\Z}\e(n^2\tau+2nz) $$
and
$$\theta_\kappa(\tau)=\theta_\kappa(\tau,0),\;\;\;\kappa=0,1. $$
It follows from  the definition of $T_{L,N}$ that
$$(T_{L,N}(h))_\nu=\sum_{\kappa\in S'/S}h_{(\kappa,\nu)}\theta_\kappa. $$
Let $\tilde{\phi}_{-2,1},\;\tilde{\phi}_{0,1}$ be the weak Jacobi forms defined in the book \cite{EZ} p.108. These functions can be written as \begin{align*}&\tilde{\phi}_{-2,1}(\tau,z)=
\psi_0(\tau)\,\theta_0(\tau,z)+\psi_1(\tau)\,\theta_1(\tau,z),\\ &\tilde{\phi}_{0,1}(\tau,z)=
\varphi_0(\tau)\,\theta_0(\tau,z)+\varphi_1(\tau)\,\theta_1(\tau,z)
 \end{align*}
 where \begin{align}\label{eqn: fourier psi phi ++}&\psi_0=-2 - 12q - 56q^2 - 208q^3+\cdots,\\ &\psi_1=q^{-1/4} + 8q^{3/4} + 39q^{7/4} + 152q^{11/4} +\cdots \notag \\ & \notag\\ &\varphi_0=10 + 108q + 808q^2 + 4016q^3+\cdots ,\notag \\ &\varphi_1= q^{-1/4} - 64q^{3/4} - 513q^{7/4} - 2752q^{11/4} 
+\cdots \notag .\end{align}
The vector-valued functions $(\psi_0,\psi_1)$ and  $(\varphi_0,\varphi_1)$ belong to the spaces $M_{-5/2}^!(\rho_S)$ and $M_{-1/2}^!(\rho_S)$ respectively, and they satisfy
\begin{align}\label{eqn: phi 0 12 ++} \tilde{\phi}_{-2,1}(\tau,0)&=\psi_0(\tau)\,\theta_0(\tau)+\psi_1(\tau)\,\theta_1(\tau)=0,\\
 \tilde{\phi}_{0,1}(\tau,0)&=\varphi_0(\tau)\,\theta_0(\tau)+\varphi_1(\tau)\,\theta_1(\tau)=12.\notag
 \end{align}

First, we construct a function $g\in M_{1/2}^{\,!}(\rho_L)$ that satisfies conditions (i) and (iii). Define
\begin{equation}\label{h++}g_{(\kappa,\nu)}:=\frac{1}{12}\varphi_\kappa f_\nu,\quad(\kappa,\nu)\in S'/S\times N'/N. \end{equation} This function satisfies
\begin{align*}T_{L,N}(g)=&\frac{1}{12}\sum_{\nu\in N'/N} e_\nu\,( g_{(0,\nu)}\theta_0+g_{(1,\nu)}\theta_1)\\ =&\frac{1}{12}\sum_{\nu\in N'/N} e_\nu\, f_\nu(\varphi_0\theta_0+\varphi_1\theta_1)\\
=&f.\end{align*}

Next, we will add a correction term to $g$ and construct a function that satisfies also (ii). Fix an integer $s>0$.
Our next goal is to construct a supplementary function $\tilde{g}(\tau)=\sum_{\lambda_\in\Z/2D\Z}\sum_{d\equiv\lambda^2(\mathrm{mod}\;4D)}\tilde{a}(d)\,\e(d\tau)\in M^{\,!}_{1/2}(\rho_L)$ with the following properties:
\begin{align}\label{eqn: tilde g 1}&\tilde{a}(-Ds^2)\neq 0\mbox{ and }\tilde{a}(-Dr^2)=0\mbox{ for all }r>s,\\ \label{eqn: tilde g 2}&
T_{L,N}(\tilde{g})=0\\ \label{eqn: tilde g 3}&\tilde{g}\mbox{ has rational Fourier coefficients.}\end{align}
To this end we consider the following theta function
   $$\widetilde{\Theta}:=\sum_{\nu\in\Z/D\Z}e_\nu\sum_{a\in\ring+\nu/\sqrt{-D}}(a^2+\overline{a}^2)\,\e(a\overline{a}\tau).
$$    By Theorem 4.1 in \cite{Bo1} his theta function belongs to $S_3(\rho)$. We define \begin{equation}\label{h++}\tilde{g}_{(\kappa,\nu)}:=\psi_\kappa\, \widetilde{\Theta}_{\nu}\,j^{\frac{s^2-t}{4}+1},\quad (\kappa,\nu)\in S'/S\times N'/N, \end{equation}
 where $$t=\begin{cases}0 &\mathrm{if}\; s\equiv0\;\mathrm{mod}\;2,\\1 \;&\mathrm{otherwise},\end{cases} $$ and $j$ is the $j$-invariant.
 First we check that the function $\tilde{g}$ satisfies condition \eqref{eqn: tilde g 1}. For $D\neq3$ we have $$\Theta_0=4q+O(q^2),\quad q=\e(\tau).$$ Hence, from \eqref{h++} we find that for $s$ even
 \begin{align*}\tilde{g}_{(0,0)}=& -8q^{-s^2/4}+O(q^{-s^2/4+1}),\\
  \tilde{g}_{(1,0)}=&4q^{-s^2/4-1/4}+O(q^{-s^2/4+3/4}),\end{align*}
 and for $s$ odd
 \begin{align*}\tilde{g}_{(0,0)}=& -8q^{-s^2/4+1/4}+O(q^{-s^2/4+5/4}),\\
  \tilde{g}_{(1,0)}=&4q^{-s^2/4}+O(q^{-s^2/4+1}).\end{align*} This proves \eqref{eqn: tilde g 1}.
  The function $\tilde{g}$ satisfies
\begin{align*}T_{L,N}(\tilde{g})=&\sum_{\nu\in N'/N} e_\nu\,( \tilde{g}_{(0,\nu)}\theta_0+\tilde{g}_{(1,\nu)}\theta_1)\\ =&\sum_{\nu\in N'/N} e_\nu\, \widetilde{\Theta}_{\ring+\nu}\,j^{\frac{s^2-t}{4}+1}\,(\psi_0\theta_0+\psi_1\theta_1)\\
=&\,0.\end{align*} This proves \eqref{eqn: tilde g 2}. The property \eqref{eqn: tilde g 3} is obvious.

  By subtracting from $g$ a suitable linear combination of functions $\tilde{g}$ for different $s$  we find a function
$$h(\tau)=\sum_{\lambda\in\Z/D\Z}e_\lambda\sum_{d\equiv\lambda^2(\mathrm{mod}\;4D)}b(d)\,
\e\Big(\frac{d}{4D}\tau\Big)\in M^{\,!}_{1/2}(\rho_L)$$ such that
\begin{align}\label{eqn: tilde h 1}&b(-Dr^2)=0\mbox{ for all }r\in\Z\backslash 0,\\ \label{eqn: tilde h 2}&
T_{L,N}(h)=f,\\ \label{eqn: tilde h 3}&h(\tau)\mbox{ has rational Fourier coefficients.}\end{align}
The final step is to show that $b(0)=0$. Identity \eqref{eqn: tilde h 2} implies that
 $$h_{(0,0)}\theta_0+h_{(1,0)}\theta_1=f_0. $$ Hence, the constant terms of these functions are equal. By the assumptions of the theorem
$$\mathrm{CT}(f_0)=0.$$ On the other hand
$$\mathrm{CT} (h_{(0,0)}\theta_0+h_{(1,0)}\theta_1)=\sum_{s\in\Z} b(-Ds^2)=b(0). $$ Thus, the function $h$ satisfies the conditions (i)-(iii) of the theorem. This finishes the proof.
\end{proof}
We observe that the Grassmanian  $\mathrm{Gr}^+(L)$ is isomorphic to the upper half-plane $\h$.
 There is a map $\h\to \mathrm{Gr}^+(L)$ given by \begin{equation} \label{eqn: D  z to v+}z\quad\to\quad v^+(z):=\Re\left(\vcenter{\xymatrix@R=0pt@C=0pt@W=0pt{z^2&z\\ z& 1}}\right)\R+\Im\left(\vcenter{\xymatrix@R=0pt@C=0pt@W=0pt{z^2&z\\ z& 1}}\right)\R \subset L\otimes\R.\end{equation}
 The group $\Gamma_0(D)$ acts on $L'$ and fixes all the elements of $L'/L$.
Denote by $X_0(D)$ the modular curve  $\overline{\Gamma_0(D)\backslash \h}$.

  Suppose that the vector $m\in L'$, the lattice $N$ and the point $\mathfrak{z}_m\in\h$ are defined as in Theorem \ref{thm: L cap m bot}. Let $h$ be the modular form $h\in S^{\,!}_{1/2}(\rho_L)$ satisfying \begin{equation}\label{eqn: TPN(h)=f 1}T_{L,N}(h)=f,\end{equation}
     that was constructed in the previous theorem. It follows from \eqref{eqn: TPN(h)=f 1} and Theorem \ref{th see-saw borcherds} that
    $$\Phi_L(h,\mathfrak{z}_m)=\Phi_N(f).$$ Recall that by definition $$\Phi_N(f)=(f,\Theta_{\Bb})^\mathrm{reg}.$$ Without loss of generality we assume that $h$ has integral negative Fourier coefficients. The infinite product $\Psi(z):=\Psi_L(h,z)$ introduced in Section \ref{sec: theta} defines a meromorphic function on $X_0(D)$. Theorem \ref{thm: borcherds inf pro} in Section \ref{sec: theta} implies
    \begin{equation}\label{eqn: pet=logPsi}
    (f,\Theta_{\Bb})^\mathrm{reg}=\log|\Psi_L(h,\mathfrak{z}_m)|.\end{equation}
    It also follows from Theorem \ref{thm: borcherds inf pro} that the divisor of $\Psi_L$ is supported at Heegner points.

\section{Heights of Heegner points} \label{sec: heegner}
In this section we compute the local height pairing between Heegner divisors. These calculations are carried out  in the celebrated series of papers \cite{GZ}, \cite{GKZ}. For the convenience of the reader we recall the main steps of the computation in what follows.

First, let as recall the definition of Heegner points and the way they can be indexed by the vectors of the lattice $L'$.

For $\ell\in L'$ with $\q(\ell)<0$ denote by $\x_\ell$ the divisor $(\mathfrak{z}_\ell)-(\infty)$ on the modular curve $X_0(D)$. The divisor $\x_\ell$ is defined over the Hilbert class field of $\Q(\sqrt{D\q(\ell)})$.

For any integer $d>0$ such that $-d$ is congruent to a square modulo $4D$, choose a residue $\beta (\mod 2D)$ with $-d\equiv\beta^2 (\mod 4D)$ and consider the set $$L_{d,\beta} = \left\{ \ell=\left(\vcenter{\xymatrix@R=0pt@C=2pt@W=0pt{a/D&b/2D\\ b/2D&c}}\right)\in L'\,\Big|\; \q(\ell)=-\,\dfrac{d}{4D},\; b\equiv\beta(\mod 2D)\right\}$$ on which $\Gamma_0(D)$ acts.
 Define the Heegner divisor $$\y_{d,\beta}=\sum_{\ell\in \Gamma_0(D)\backslash L_{d,\beta}}\x_\ell .$$
The Fricke involution acts on $L'$ by $$\ell\to\frac{1}{D}\left(\vcenter{\xymatrix@R=0pt@C=2pt@W=0pt{0&1\\-D&0}}\right)\ell
 \left(\vcenter{\xymatrix@R=0pt@C=2pt@W=0pt{0&-D\\1&0}}\right)$$ and maps $L_{d,\beta}$ to $L_{d,-\beta}$. Set \begin{equation}\label{eqn: y*}\y^*_d=\y_{d,\beta}+\y_{d,-\beta}.\end{equation} The divisor $\y^*_d$ is defined over  $\Q$ (\cite{GKZ} p. 499.)

 Now we would like to compute the local height pairings between the divisor $\x_\ell$ and a Heegner divisor. The definition of the local height pairing is given in Section \ref{sec: height}. The divisors $\x_\ell$ and $\y^*_d$ have the  point  $\infty$ at their common support. In order to define the height pairing between these divisors we must fix a
uniformizing parameter $\pi$ at this cusp. We let $\pi$ denote the Tate parameter $q$ on the
family of degenerating elliptic curves near $\infty$. This is defined over $\Q$. Over $\C$
we have $q=\e(z)$ on $X_0^*(D)=\Gamma^*_0(D)\backslash\overline{\h}$, where $z\in\h$ with $\Im(z)$ sufficiently large. The following theorem can be deduced from the computations in Section IV.4 in \cite{GKZ}.
 \begin{theorem}\label{thm: height} Let $d_1,\;d_2>0$ be two integers and $\beta_1,\;\beta_2$ be two elements of $\Z/2D\Z$ with $-d_1\equiv\beta_1^2 (\mod 4D)$ and $-d_2\equiv\beta_2^2 (\mod 4D)$. Suppose that $d_1$ is fundamental and $d_2/d_1$ is not a full square. 
 Fix a vector $\ell\in L_{d_1,\beta_1}$.  Let $p$ be a prime with  $\mathrm{gcd}(p,D)=1$. Choose a prime ideal  $\Pp$  lying above $p$ in the Hilbert class field of $\Q(\sqrt{-d_1})$. Then the following formula for the local height holds:\\
 in the case $\big(\frac{p}{d_1}\big)=1$ we have
 \begin{equation}\label{eqn: height +1}\langle \x_\ell, \y^*_{d_2}\rangle_\Pp=0,\end{equation}
 in the case $\big(\frac{p}{d_1}\big)=-1$ we have
 \begin{equation}\label{eqn: height -1}\langle \x_\ell, \y^*_{d_2}\rangle_\Pp=\log(p)\!\!\!\!\!\!\sum_{\substack{ r\in\Z\\ \scriptstyle r\equiv\beta_1\beta_2\;\mathrm{mod}\;2}}\!\!\!\!\!\!\delta_{d_1}(r)\,r_{\mathfrak{n}\overline{\ccc}^2\aaa^2}
 \bigg(\frac{d_1d_2-r^2}{4Dp}\bigg)\,\mathrm{ord}_p\bigg(\frac{d_1d_2-r^2}{4D}\bigg) .\end{equation}
 Here $\ccc=\Z \mathfrak{z}_\ell+\Z$, $\mathfrak{n}=\Z D+\Z\frac{\beta_1+\sqrt{-d_1}}{2}$, $\aaa$ is any ideal in the ideal class $\Aa$ defined by \eqref{eqn: wp A}, and $$\delta_d(r)=\begin{cases}2\quad\mbox{for}\;r\equiv 0\;\mathrm{mod}\;d;\\ 1\quad\mbox{otherwise}.\end{cases}$$
 \end{theorem}
 \begin{proof}

 The curve $X_0(D)$ may be described over $\Q$ as the compactification of the space of moduli of elliptic curves with a cyclic subgroup of order $D$ \cite{GZ}. Over a field $k$ of characteristic zero, the points $\y$ of $X_0(D)$ correspond to diagrams
 $$\psi:F\to F' ,$$ where $F$ and $F'$ are (generalized) elliptic curves over $k$ and $\psi$ is an isogeny over $k$ whose kernel is isomorphic to $\Z/D\Z$ over an algebraic closure $\overline{k}$.

 The point $\mathfrak{z}_\ell\in\h$ defines the point $\x\in X_0(D)$. Then $\x=(\phi:E\to E')$ and over $\C$ this diagram is isomorphic to $$\xymatrix{\C/\ccc\ar[r]^{\mathrm{id}_\C}&\C/\ccc\mathfrak{n} }.$$

 Following the calculations in \cite{GZ} we reduce the computation of local heights to a problem in arithmetic intersection theory. Let us set up some notations.  Denote by $v$ the place of $H_{d_1}$, the Hilbert class field of $\Q(\sqrt{-d_1})$, corresponding to prime ideal $\Pp$. Denote by $\Lambda_v$ the ring of integers in the completion $H_{d_1,v}$ and let $\pi$ be an uniformizing parameter in $\Lambda_v$. Let $W$ be the completion of the maximal unramified extension extension of $\Lambda_v$. Let $\underline{X}$ be a regular model  for $X$ over $\Lambda_v$ and $\underline{\x}$, $\underline{\y}$ be the sections of $\underline{X}\otimes \Lambda_v$ corresponding to the points $\x$ and $\y$. A model that has a modular interpretation is described in \cite{GZ} Section III.3). The general theory of local height pairing \cite{Gross} implies
 $$\langle \x,\y \rangle_v=-(\underline{\x}\cdot\underline{\y})\log{p}. $$
 The intersection product is unchanged if we extend scalars to $W$.
 By Proposition 6.1 in \cite{GZ}
 $$ (\underline{\x}\cdot\underline{\y})_W=\frac 12\sum_{n\geq1}\mathrm{Card}\mathrm{Hom}_{W/\pi^n}(\underline{\x},\underline{\y})_{\mathrm{deg}1}.$$

 Denote by $R$ the ring $\mathrm{Hom}_{W/\pi}(\underline{\x_\ell})$.  On p. 550 of \cite{GKZ} the following formula for the intersection number is obtained  
 \begin{equation}\label{eqn: height W} (\underline{\x_\ell}\cdot\underline{\y_{d_2}^*})_W=\frac14\!\!\!\!\!\!\sum_{\xymatrix@R=0pt{\scriptstyle r^2<d_1d_2\\ \scriptstyle r\equiv\beta_1\beta_2(\mathrm{mod}\;2D)}}\!\!\!\!\!\!\mathrm{Card}\big\{S_{[d_1,2r,d_2]}\to R \mod R^\times \big\}\,\mathrm{ord}_p\bigg(\frac{d_1d_2-r^2}{4D}\bigg),\end{equation}
where  $S_{[d_1,2r,d_2]}$ is the Clifford order
$$ S_{[d_1,2r,d_2]}=\Z+\Z\frac{1+e_1}{2}+\Z\frac{1+e_2}{2}+\Z \frac{(1+e_1)(1+e_2)}{4}, $$
$$e_1^2=-d_1,\quad e_2^2=-d_2,\quad e_1e_2+e_2e_1=2r.$$

In the case $\big(\frac{p}{d_1}\big)=1$ the ring $R$ is isomorphic to an order in $\ring_{d_1}$. Since $d_1/d_2$ is not a full square the ring $R$ can not contain  the Clifford order $S_{[d_1,2r,d_2]}$. Hence, $(\underline{\x_\ell}\cdot\underline{\y_{d_2}^*})_W=0$ and this proves \eqref{eqn: height +1}.

Now we consider the case $\big(\frac{p}{d_1}\big)=-1$. Formula (9.3) in \cite{GZ} gives us a convenient description of the ring $R$. Namely, for $a,b\in\Q(\sqrt{-d_1})$ denote $$[a,b]=\left(\vcenter{\xymatrix@R=0pt@C=0pt@W=0pt{a&b\\p\overline{b}& \overline{a}}}\right)$$ and consider the quaternion algebra over $\Q$
 $$B=\Big\{[a,b]\,\Big|\,a,b\in\Q(\sqrt{-d_1})\Big\} .$$
 Then $R$ is an Eichler order of index $D$ in this quaternion algebra and it is given by
 $$R=\Big\{[a,b]\,\Big|\,a\in\mathfrak{d}^{-1},\;b\in\mathfrak{d}^{-1}\mathfrak{n}\overline{\aaa}\overline{\ccc}\aaa^{-1}\ccc^{-1}, a\equiv b\mod\ring_{d_1}\Big\}, $$ where $\mathfrak{d}$ is the different of $\Q(\sqrt{-d_1})$.

 By the same computations as in Lemma 3.5 of \cite{GZsingmoduli} we find that the number of embeddings of $S_{[d_1,2r,d_2]}$ into $R$, normalized so that the image of $e_1$ is $[\sqrt{-d_1},0]$, is equal to  $$\delta_{d_1}(r)\,r_{\mathfrak{n}\overline{\ccc}^2\aaa^2}\bigg(\frac{d_1d_2-r^2}{4Dp}\bigg)\,
\mathrm{ord}_p\bigg(\frac{d_1d_2-r^2}{4D}\bigg).$$This finishes the proof of the theorem.
\end{proof}

\section{Proof of Theorem 1.}\label{sect: proof of prime my}

\noindent\emph{Proof of Theorem \ref{thm: prime my}.}
  Since the discriminant $-D$ is prime, the class number of $K$ is odd and there exists an ideal class $\Cc$ such that $\Bb=\overline{\Cc}^2$ in the ideal class group. The class $\Cc$ contains an ideal of the form  \begin{equation}\label{eqn: Cc z}\ccc=\mathfrak{z}\Z+\Z,\end{equation} where $\mathfrak{z}$ is a CM point of discriminant $-D$. Property \eqref{eqn: Cc z} is preserved when we act on $\mathfrak{z}$ by elements of $\mathrm{SL}_2(\Z)$. As we have explained in the proof of  Theorem \ref{thm: L cap m bot}, we may assume  that $\mathfrak{z}$ satisfies the quadratic equation
 $$a\mathfrak{z}^2+b\mathfrak{z}+c=0 $$
 for $a\in D\Z,b\in D\Z,c\in\Z$ and $b^2-4ac=-D$. The matrix $$m=\frac1D\left(\vcenter{\xymatrix@R=0pt@C=2pt@W=0pt{c&-b/2\\-b/2&a}}\right)$$ belongs to the lattice $L'$ and has the norm $-1/4$. Lemma \ref{lemma: L cap m bot} implies that the lattice $N:= L \cap m^ \bot$ corresponds to the fractional ideal $\ccc^2$ as explained in Section \ref{sec: theta} and moreover, the lattice $L$ splits as  $L= N\oplus 2m\Z$.

 Next, by Theorem \ref{thm: trick pro} we find a weak cusp form $h\in S^{\,!}_{1/2}(\rho_L)$  satisfying \begin{equation}\label{eqn: TPN(h)=f}T_{L,N}(h)=f,\end{equation}
    where $T_{L,N}$ is defined as in Theorem \ref{th see-saw borcherds}.
Function $h$ has the Fourier expansion of the form
    $$h(\tau)=\sum_{\beta\in\Z/2D\Z}e_\beta\sum_{d\equiv \beta^2 (\mod\; 4D)}b(d)\,\e\Big(\frac{d}{4D}\tau\Big). $$
It follows from \eqref{eqn: TPN(h)=f} and Theorem \ref{th see-saw borcherds} that
    $$\Phi_N(f)=\Phi_L(h,\mathfrak{z}).$$
    From Theorem \ref{thm: borcherds inf pro} in Section \ref{sec: theta} we know that
    \begin{equation}\label{eqn: Phi=logPsi}
    \Phi_L(h,\mathfrak{z})=\log|\Psi_L(h,\mathfrak{z})|,\end{equation}
    where $\Psi(z)=\Psi_L(h,z)$ is a meromorphic function. Theorem \ref{thm: borcherds inf pro} also implies that
    \begin{equation}\label{eqn: divPsi}
    \mathrm{div}(\Psi)=\sum_{d=0}^\infty b(-d)\,\y^*_d,\end{equation}
    where $\y_d^*$ is the Heegner divisor defined in \eqref{eqn: y*}.

     Set $\x=(\mathfrak{z})-(\infty)$. The condition (ii) of Theorem \ref{thm: trick pro} implies that the function $\Phi_L(h,\cdot)$ is real analytic at point $\mathfrak{z}$. Thus, the only point in the common support of $\x$ and $\mathrm{div}(\Psi)$ is $\infty$. Recall, that we have fixed the uniformizing parameter $\pi$ at this cusp to be the Tate parameter $q$ on the
family of degenerating elliptic curves near $\infty$.

 Recall that the divisors $\x$ and $\mathrm{div}(\Psi)$ are defined over $H$. The axioms of local height (listed in Section \ref{sec: height}) together with the refined definition \eqref{eqn: common support} imply that for each prime $\Pp$ of $H$
    \begin{equation}\label{eqn: height div}
    \mathrm{ord}_\Pp\big(\Psi(\mathfrak{z})\big)\log p-\mathrm{ord}_\Pp\big(\Psi[\infty]\big)\log p=\big\langle \x, \sum_{d=1}^\infty b(-d)\,\y^*_d\big\rangle_\Pp.\end{equation} From the infinite product of Theorem 13.3 in \cite{Bo1} we find that $\Psi[\infty]=1$ for the choice of the uniformizing  parameter at $\infty$ as above. Theorem~\ref{thm: trick pro} part~(ii) implies that $d/D$ is not a full square provided $b(-d)\neq0$. Thus,
    by Theorem~\ref{thm: height} for each prime $\Pp$ of $H$ lying above a rational prime $p$ with $\big(\frac pD\big)\neq 0$  we obtain
    $$ \langle \x, \y^*_d\rangle_\Pp=0$$ in the case $\big(\frac pD\big)=1$, and
    \begin{equation}\label{eqn: height}\langle \x, \y^*_d\rangle_\Pp=\log(p)\!\!\!\!\!\!\sum_{\substack{ n\in\Z\\ n\equiv d(\mathrm{mod}\;2)}}\!\!\!\!\!\!r_{\overline{\Cc}^2\Aa^2}\bigg(\frac{d-Dn^2}{4p}\bigg)\,
 \mathrm{ord}_p\bigg(\frac{d-Dn^2}{4}\bigg) \end{equation}
 in the case $\big(\frac pD\big)=-1$.
  We observe that the sum $$\sum_{d=0}^\infty b(-d) \!\!\!\!\!\!\sum_{\substack{ n\in\Z\\ n\equiv d(\mathrm{mod}\;2)}}\!\!\!\!\!\!r_{\overline{\Cc}^2\Aa^2}\bigg(\frac{d-Dn^2}{4p}\bigg)\,
 \mathrm{ord}_p\bigg(\frac{d-Dn^2}{4}\bigg)$$ is equal to the constant term with respect to $\e(\tau)$ of the following series
 $$ \sum_{\nu\in\Z/D\Z}\bigg(\Big( h_{(0,\nu)}\theta_0+h_{(1,\nu)}\theta_1\Big) \sum_{t\equiv\nu\;\mathrm{mod}\;D}r_{\Bb\Aa^2}\Big(\frac t p\Big)\,\mathrm{ord}_p(t)\,\e\Big(\frac t D\tau\Big)\;\bigg). $$
    The equation \eqref{eqn: TPN(h)=f} implies
    \begin{equation}\label{eqn:f=h theta} f_\nu=h_{(0,\nu)}\theta_0+h_{(1,\nu)}\theta_1,\quad \nu\in\Z/D\Z.\end{equation}
    Hence, combining the equations \eqref{eqn: height} and \eqref{eqn:f=h theta} we arrive at
    $$\Big\langle \x, \sum_{d=0}^\infty b(-d)\,\y^*_d\Big\rangle_\Pp=\log{p}\sum_{\nu\in\Z/D\Z}\sum_{t=0}^\infty c_\nu(-t)\,r_{\Bb\Aa^2}\Big(\frac{t}{p}\Big)\,\mathrm{ord}_p(t). $$
Finally, the equations \eqref{eqn: Phi=logPsi} and \eqref{eqn: height div} imply
$$\mathrm{ord}_\Pp(\alpha)=\mathrm{ord}_\Pp(\Psi_L(h,\mathfrak{z}))=\frac{1}{\log {p}}\,\Big\langle \x, \sum_{d=0}^\infty b(-d)\,\y^*_d\Big\rangle_\Pp= $$
$$=\sum_{\nu\in\Z/D\Z}\sum_{t=0}^\infty c_\nu(-t)\,r_{\Bb\Aa^2}\Big(\frac{t}{p}\Big)\,\mathrm{ord}_p(t).$$
 This finishes the proof of Theorem \ref{thm: prime my}. $\Box$
\section{Numerical example}
In this section we illustrate the Theorem \ref{thm: prime my} with the following numerical example coming from the evaluation of higher Green's functions at CM-points. In particular, we prove identity \eqref{eqn: example} stated in the introduction.

Recall that the higher Green's functions are real-valued functions of two variables on the upper half-plane which are bi-invariant under the action of  $SL_2(\Z)$, have a logarithmic
singularity along the diagonal and satisfy $\Delta f = k(1 - k)f$, where $k$ is a positive integer. The precise definition of these functions can be found in \cite{Via CM Green}. We denote higher Green's functions by $G_k(z_1,z_2)$. In our recent work \cite{Via CM Green} we have related the CM-values of higher Green's functions and regularized products of weight one modular forms. Let us explain this relation on the following concrete example.

We consider the pair of CM points $$\mathfrak{z}_1=\frac{1+\sqrt{-23}}{4},\quad\mathfrak{z}_2=\frac{-1+\sqrt{-23}}{4} $$ lying in the imaginary quadratic field $K=\Q(\sqrt{-23})$. In the last section of \cite{GZ} B. Gross and D. Zagier have conjectured that higher Green's functions have "algebraic" CM values. In particular, Conjecture (4.4) of \cite{GZ} predicts that for $k=2,3,4,5$ and $7$
\begin{equation}\label{eqn: CM Green} G_k(\mathfrak{z}_1,\mathfrak{z}_2)=23^{1-k}\,\log|\alpha_k|, \end{equation}
where $\alpha_k$ are some algebraic numbers lying at the Hilbert class field of $K$.

The connection between the CM values \eqref{eqn: CM Green} and regularized Petersson products is as follows. Note that the space $S_{2k}$ of cusp forms of weight $2k$ on $\slz$ is zero precisely when $k=1,2,3,4,5$ and $7$. For these values of $k$ the Serre duality implies that there exists the unique modular form $g_k$ in the space $M^!_{2-2k}$ of weakly holomorphic modular forms of weight $2-2k$ on  $\slz$ with the Fourier expansion $g_k=q^{-1}+O(1)$. Denote by $\Bb$ and $\Cc$, respectively, the ideal classes of $K$ containing the fractional ideals $\Z\mathfrak{z}_1+\Z$ and  $\Z\mathfrak{z}_2+\Z$, respectively. Theorem 4 in \cite{Via CM Green} implies that
$$G_k(\mathfrak{z}_1,\mathfrak{z}_2)=\big([g_k,\Theta_{\Bb\overline{\Cc}}]_{k-1},\Theta_{\Bb\Cc}\big)^{\mathrm{reg}},$$
where $[\cdot,\cdot]_{k-1}$ denotes the $(k-1)$-st Rankin-Cohen brackets. Thus, Theorem~1 implies the Conjecture (4.4) formulated in \cite{GZ} and, moreover, gives the factorization formula for the CM-values of higher Greens functions.

  At the rest of this section we will compute the regularized integral \eqref{eqn: example}, which is equal to the CM-value $G_2(\mathfrak{z}_1,\mathfrak{z}_2)$. The function $f$ defined by \eqref{eqn: f example} has the Fourier expansion of the form $$f(\tau)=\sum_{\nu\in \Z/D\Z}e_\nu\,\sum_{\substack{t\equiv \nu^2(\mathrm{mod}\;D)\\t\gg-\infty}} c(t)\,\e\Big(\frac t D\tau\Big). $$
The negative Fourier coefficients of $f$ are given in the following table:
\begin{center}
\begin{tabular}{|c|c|c|c|c|c|c|c|c|c|c|}\hline\rule{0pt}{15pt} $t$&$5$&$7$&$11$&$14$&$15$&$17$&$19$&$20$&$21$&$23$  \\[3pt] \hline\rule{0pt}{15pt}
$c(-t)$&$26$&$18$&$2$&$-5$&$-7$&$-11$&$-15$&$-17$&$-19$&$-23$\\[3pt]
\hline
\end{tabular}
\end{center}

Denote by $N$ an even lattice $(\ring_{K},\q)$, where the quadratic form $\q$ is defined as $\q(x)=\mathrm{N}_{K/\Q}(x)$ for $x\in\ring_K$.

Consider the lattice \begin{equation}\label{eqn: lattice L23}L=\bigg\{\left(\vcenter{\xymatrix@R=0pt@C=2pt@W=0pt{A/23&B\\B&C}}\right) \Big|A,B,C\in\Z \bigg\}\end{equation} equipped with the quadratic form $\q(\ell):=-23\det(\ell)$. Choose the vector $$m= \left(\vcenter{\xymatrix@R=0pt@C=2pt@W=0pt{6/23&1/2\\1/2&1}}\right). $$
The vector $m$ has the norm $\q(m)=-1/4$ and by Lemma \ref{lemma: L cap m bot} its orthogonal complement  $L\cap m^\bot$ is isomorphic to $N$. Moreover, $L$ splits into a direct sum $L\cong 2m\Z\oplus N.$ From \eqref{eqn: frakz} we find $$ \mathfrak{z}_m=\frac{23+\sqrt{-23}}{46}.$$
Note that the fractional ideal $$\ccc:=\Z+ \mathfrak{z}_m\Z=(\sqrt{-23})^{-1}$$ is principal.

Recall that the group $\Gamma_0^*(23)$ acts on $L$ by isometries and the map \eqref{eqn: D  z to v+} gives an isomorphism between $\Gamma_0^*(23)\backslash\mathrm{Gr}^+(L)$ and $\Gamma_0^*(23)\backslash\h$. We denote by $X^*_0(23)$ the modular curve $\overline{\Gamma_0^*(23)\backslash\mathrm{Gr}^+(L)}$.

Let $\phi, \psi, \widetilde{\Theta}$ and $j$ be as in the proof of Theorem \ref{thm: trick pro}. By Theorem \ref{thm: trick pro} we find that the vector valued function $ h$ given componentwise by
$$h_{(\nu,\kappa)}= \varphi_\kappa f_\nu+\frac{10}{92}\psi_\kappa\widetilde{\Theta}_{\nu}\,j^{2}-
\frac{7416}{46}\psi_\kappa\widetilde{\Theta}_{\nu}\,j,\;(\kappa,\nu)\in S'/S\times N'/N,$$
belongs to $S^!_{1/2}(\rho_L)$ and satisfies conditions (i)-(iii). Moreover, the function $12\cdot 23^2\cdot h$ has integral Fourier coefficients. By Theorem \ref{th see-saw borcherds} these conditions imply that
$$(f,\Theta_{\mathcal{O}})^{\mathrm{reg}}=\Phi_L(h,v^+(\mathfrak{z}_m)). $$
Function $h$ has the Fourier expansion $$h(\tau)=\frac{1}{12\cdot 23^2} \sum_{\lambda\in\Z/46\Z}e_\lambda\sum_{d\equiv\lambda^2(\mathrm{mod}\;92)}b(d)\,
\e\Big(\frac{d}{92}\tau\Big) .$$ The negative Fourier coefficients of $12\cdot 23^2\cdot h$ are given in the following table
\begin{center}
\begin{tabular}{cccc}
\begin{tabular}{|c|c|}\hline\rule{0pt}{15pt}
$d$&$b(-d)$\\ \hline
$7$&$-3126678$\\
$11$&$1455$\\
$15$&$2497$\\
$19$&$-783263$\\
$20$&$-884$\\
$28$&$-1228$\\
$40$&$-790$\\
$43$&$884$\\
$44$&$-68$\\
$51$&$990$\\
$56$&$-792$\\
\hline
\end{tabular}
&
\begin{tabular}{|c|c|}\hline\rule{0pt}{15pt}
$d$&$b(-d)$\\ \hline
$60$&$616$\\
$63$&$431$\\
$67$&$68$\\
$68$&$968$\\
$76$&$-352$\\
$79$&$426$\\
$80$&$-630$\\$
83$&$-462$\\
$84$&$-630$\\
$88$&$332$\\
$91$&$-726$\\
\hline
\end{tabular}
&
\begin{tabular}{|c|c|}\hline\rule{0pt}{15pt}
$d$&$b(-d)$\\ \hline
$99$&$36$\\
$103$&$111$\\
$107$&$87$\\
$111$&$-156$\\
$112$&$-130$\\
$115$&$-276$\\
$120$&$-160$\\
$135$&$65$\\
$136$&$-10$\\
$143$&$80$\\
$148$&$-90$\\
\hline
\end{tabular}
&
\begin{tabular}{|c|c|}\hline\rule{0pt}{15pt}
$d$&$b(-d)$\\ \hline
$152$&$70$\\
$159$&$5$\\
$160$&$110$\\
$168$&$-40$\\
$171$&$45$\\
$175$&$-35$\\
$180$&$-10$\\
$183$&$-55$\\
$191$&$20$\\
$203$&$5$\\
&\\
\hline
\end{tabular}
\end{tabular}
\end{center}

 Since the function $12\cdot 23^2\cdot h$ has integral negative Fourier coefficients, the infinite product $\Psi(z):=\Psi_L(12\cdot 23^2\cdot h,z)$ introduced in Section \ref{sec: theta} defines a meromorphic function on $X_0^*(23)$ with the only zeroes and poles at Heegner points. Theorem \ref{thm: borcherds inf pro} in Section \ref{sec: theta} implies
    \begin{equation}\label{eqn: pet=logPsi}
    (f,\Theta_{\mathcal{O}})^\mathrm{reg}=\frac{1}{12\cdot23^2}\log|\Psi_L(12\cdot 23^2\cdot h,\mathfrak{z}_m)|.\end{equation}

The curve $X_0^*(23)$ has genus $0$ and only one cusp. Let $j^*_{23}(z)$ be the Hauptmodul for $\Gamma_0^*(23)$ having the Fourier expansion $j^*_{23}(z)=q^{-1}+O(q)$, where $q=\e(z)$. This function is given explicitly by
\begin{align*} j^*_{23}(z)&=\frac{1}{\eta(z)\eta(23z)}\sum_{a,b\in\Z}\e((a^2+ab+6b^2)z)-3\\
&= q^{-1} + 4q + 7q^2 + 13q^3 + 19q^4 + 33q^5 + 47q^6 + 74q^7 + \cdots. \end{align*}

 For any integer $d>0$ such that $-d$ is congruent to a square modulo $92$, choose an integer $\beta (\mod 46)$ with $-d\equiv\beta^2 (\mod 92)$ and consider the set $$L_{d,\beta} = \left\{ \ell=\left(\vcenter{\xymatrix@R=0pt@C=2pt@W=0pt{a/23&b/46\\b/46&c}}\right)\in L'\,\Big|\; \q(\ell)=-d/92, b\equiv\beta(\mod 46)\right\}$$ on which $\Gamma_0(23)$ acts. The Fricke involution acts on $L'$ by $$x\to\frac{1}{23}\left(\vcenter{\xymatrix@R=0pt@C=2pt@W=0pt{0&1\\-23&0}}\right)p
 \left(\vcenter{\xymatrix@R=0pt@C=2pt@W=0pt{0&-23\\1&0}}\right)$$ and maps $L_{d,\beta}$ to $L_{d,-\beta}$.

We define a polynomial $\mathcal{H}_{d,23}(X)$ by
$$\mathcal{H}_{d,23}(X)= \prod_{\ell\in L_{d,\beta}}(X-j^*_{23}(\mathfrak{z}_\ell))^{1/|\mathrm{Stab}(\ell)|}.$$
It follows from Theorem B3 part 2 that
$$\Psi(z,12\cdot 23^2\cdot h)=\prod_{d\ll\infty} \mathcal{H}_{d,23}(j^*_{23}(z))^{b(-d)}, $$
 and the numbers $b(-d)$ are given in the above table.

Recall that $K$ denotes the imaginary quadratic field $\Q(\sqrt{-23})$ and $H$ denotes its Hilbert class field. Let the algebraic numbers $\rt,\pp_l,\qq_l\in H$ be defined as in \eqref{eqn: p_l p_l^2}.
    The value of the Hauptmodul $j^*_{23}$ at the point $\mathfrak{z}_m=\frac{23+\sqrt{-23}}{46}$ is equal to $-\rt-2.$ The values of $\mathcal{H}_{d,23}\Big(j^*\big(\frac{23+\sqrt{-23}}{46}\big)\Big)$ for small values of $d$ are given in the following table.

\begin{center}
   \begin{tabular}{|l|l|l|}\hline\rule{0pt}{15pt}
$d$   &  $\mathcal{H}_{d,23}(X)$&$\mathcal{H}_{d,23}(-2-\rt)$ \\[2pt]  \hline\rule{0pt}{15pt}
$7$  &  $(X + 2)^2$&$\rt^{2}\,$ \\[2pt]
$11$ &  $(X + 1)^2$&$\rt^{6}\,$ \\[2pt]
$15$ &  $(X^2 + 3X + 3)^2$&$\rt^{10}\,$ \\[2pt]
$19$ &  $(X + 3)^2$&$\rt^{-8}\,$ \\[2pt]
$20$ &  $(X^2 + 4X + 5)^2$&$\pp_5^{2}\,\rt^{10}\,$ \\[2pt]
$28$ &  $X^2(X + 2)^2$&$\pp_7^{2}\,\rt^{2}\,$ \\[2pt]
$40$ &  $(X^2 + 2X + 3)^2$&$\pp_{25}^{2}\,\rt^{6}\,$ \\[2pt]
$43$ &  $(X - 1)^2$&$\pp_5^{4}\,\rt^{16}\,$ \\[2pt]
$44$ &  $(X + 1)^2(X^3 + 7X^2 + 17X + 13)^2$&$\pp_{11}^{2}\,\rt^{10}\,$ \\[2pt]
$51$ &  $(X^2 + 4X + 7)^2$&$\pp_7^{4}\,\rt^{-6}\,$ \\[2pt]
$56$ &  $(X^4 + 4X^3 - 16X - 17)^2$&$\pp_{49}^{2}\,\rt^{12}\,$ \\[2pt]
$60$ &  $(X^2 + 3X + 3)^2(X^2 + 7X + 13)^2$&$\pp_{25}^{2}\,$ \\[2pt]
$63$ &  $(X + 2)^2(X^4 + 5X^3 + 12X^2 + 20X + 19)^2$&$\pp_{25}^{4}\,\rt^{8}\,$ \\[2pt]
$67$ &  $(X - 3)^2$&$\pp_{11}^{4}\,\rt^{6}\,$ \\[2pt]
$68$ &  $(X^4 + 10X^3 + 34X^2 + 46X + 25)^2$&$\pp_{17}^{2}\,\rt^{-6}\,$ \\[2pt]
$76$ &  $(X + 3)^2(X^3 - X^2 - 9X - 9)^2$&$\pp_{19}^{2}\,\rt^{4}\,$ \\[2pt]
$79$ &  $(X^5 + 10X^4 + 43X^3 + 90X^2 + 90X + 27)^2$&$\pp_{49}^{4}\,\rt^{16}\,$ \\[2pt]
$80$ &  $(X^2 + 4X + 5)^2(X^4 + 6X^3 + 20X^2 + 30X + 17)^2$&$\pp_5^{2}\,\pp_{25}^{2}\,\rt^{28}\,$
\\[2pt]
$83$ &  $(X^3 - X^2 - 13X - 19)^2$&$\pp_{25}^{4}\,\rt^{6}\,$ \\[2pt]
$84$ &  $(X^4 + 2X^3 + 6X^2 + 14X + 13)^2$&$\pp_{49}^{2}\,\rt^{26}\,$ \\[2pt]
$91$ &  $(X^2 - 4X - 9)^2$&$\pp_{17}^{4}\,\rt^{-6}\,$ \\[2pt]
$99$ &  $(X + 1)^2(X^2 + 8X + 19)^2$&$\pp_{19}^{4}\,\rt^{-8}\,$ \\[2pt]
$103$ & $(X^5 + 4X^4 + 7X^3 + 33X^2 + 99X + 81)^2$&$\pp_5^{4}\,\pp_{25}^{4}\,\rt^{18}\,$ \\[2pt]
$107$ & $(X^3 + 5X^2 + 19X + 31)^2$&$\pp_{49}^{4}\,\rt^{8}\,$ \\[2pt]
$115$ & $(X + 5)^2$&$\qq_{23}^2$ \\[2pt] \hline
\end{tabular}
\end{center}
Using this values of polynomials $\mathcal{H}_{d,23}(X)$ together with values for those $d$ which are not included in the table, we finally we arrive at
$$\frac{23}{12}\log\Big|\Psi_P\Big(\frac{23+\sqrt{-23}}{46},h\Big)\Big|=\log\big|\pp_5^{\,18}\,\pp_{25}^{-42}\,
    \pp_7^{36}\,\pp_{49}^{-48}\,\pp_{11}^{4}\,\pp_{17}^{-22}\,\pp_{19}^{-30}\,\qq_{23}^{-23}
    \,\rt^{-9\cdot23}\big|.$$
This proves the result \eqref{eqn: green 23 numerical 234++}  obtained by numerical integration.

{\footnotesize
\noindent
University of Cologne, Gyrhofstrasse 8b, 50823 Cologne, Germany\\
{\it Email address: mviazovs@math.uni-koeln.de}}

\end{document}